\documentclass[11pt,a4paper]{article}

\usepackage[T1]{fontenc}
\usepackage[utf8]{inputenc}
\usepackage{lmodern}
\usepackage{amsmath, amsfonts, amssymb, amsthm}
\usepackage[left=2cm,right=2cm,top=2cm,bottom=2cm]{geometry}
\usepackage[hidelinks]{hyperref}
\hypersetup{
    pdftitle={Towards computing high-order p-harmonic descent directions and their limits in shape optimization},
    pdfauthor={Henrik Wyschka, Martin Siebenborn},
    pdfsubject={mathematics},
    pdfkeywords={shape optimization, p-Laplace},
    colorlinks=false,  			
    linkcolor=darkblue, 
    citecolor=green,         
    pdfstartview=Fit,           
    pdfpagemode=UseNone,    
}
\usepackage{graphicx}
\usepackage{float}
\usepackage{setspace}
\usepackage[format=hang]{caption}
\usepackage{subcaption}

\usepackage{titling}
\predate{}
\postdate{}
\usepackage{authblk}

\usepackage{nicefrac}

\newcommand{\R}{\mathbb{R}}
\DeclareMathOperator{\diag}{diag}
\theoremstyle{plain}
\newtheorem{theorem}{Theorem}[section]
\newtheorem{lemma}[theorem]{Lemma}
\theoremstyle{definition}
\newtheorem*{remark}{Remark}

\usepackage[style=numeric,giveninits]{biblatex}
\title{\textbf{Towards computing high-order p-harmonic\\ descent directions and their limits\\ in shape optimization}}
\author[]{Henrik Wyschka\thanks{henrik.wyschka@uni-hamburg.de}\;}
\author[]{Martin Siebenborn}
\affil[]{Department of Mathematics, University of Hamburg\protect\\ Bundesstr. 55, 20146 Hamburg, Germany}
\date{}

\begin{document}
\maketitle

\begin{abstract}
    \noindent We present an extension of an algorithm for the classical scalar $p$-Laplace Dirichlet problem to the vector-valued $p$-Laplacian with mixed boundary conditions in order to solve problems occurring in shape optimization using a $p$-harmonic approach.
    The main advantage of the proposed method is that no iteration over the order $p$ is required and thus allow the efficient computation of solutions for higher orders.
    We show that the required number of Newton iterations remains polynomial with respect to the number of grid points and validate the results by numerical experiments considering the deformation of shapes.
    Further, we discuss challenges arising when considering the limit of these problems from an analytical and numerical perspective, especially with respect to a change of sign in the source term.
\end{abstract}

\section{Introduction}
\label{sec::Introduction}

Shape optimization constrained to partial differential equations is a  vivid field of research with high relevance for industrial grade applications. 
Mathematically, we consider the problem
\begin{equation*}
    \begin{aligned}
        \min_{\Omega \in \mathcal{A}}\quad & J(v_s,\Omega)\\ 
        \text{s.t.}\quad & \mathcal{E}(v_s,\Omega) = 0
    \end{aligned}
    \label{eq:ShapeOptProblem}
\end{equation*}
where $J$ denotes the objective or shape functional depending on a state variable $v_s$ and a Lipschitz domain $\Omega \subset \R^d$, which is to minimize over the set of admissible shapes $\mathcal{A}$. 
Further, the state and the domain have to fulfill a PDE constraint $\mathcal{E}$.
A common solution technique for this type of problem is to formulate it as a sequence of deformations to the initial shape \cite{sokolowski1992}. 
Each of these transformations has the form 
\begin{equation}
    \tilde{\Omega} = (\mathrm{id}+tv)(\Omega)
    \label{eq:PerturbationOfIdentity}
\end{equation}
for a step-size $0 < t \leq 1$ and a descent vector field $v:\R^d \rightarrow \R^d$ in the sense $J^\prime(\Omega)\, v < 0$ with shape derivative denoted by $J^\prime$.
Hence, the set $\mathcal{A}$ is implicitly defined by all shapes reachable via such transformations to the initial shape.
From the analytical derivation of this procedure it is required that the $v$ are at least $W^{1,\infty}(\R^d,\R^d)$ inherently ensuring that all admissible shapes remain Lipschitz. 
However, in practice this condition is often neglected and so-called Hilbert space methods \cite{allaire2021} are used.
The most prominent example of those consists of finding $v \in W^{k,2}(\Omega,\R^d)$.
This results in smooth shapes, which are relevant for some applications, but not necessarily optimal.
Further, the mesh quality often deteriorates and it would require a high differentiability order $k$ to obtain a Lipschitz transformation from the Sobolev embedding.\\

\noindent Recent development suggests using a $p$-harmonic approach \cite{deckelnick2021} to determine descent directions. 
This technique can be obtained by considering the steepest descent in the space of Lipschitz transformations
\begin{equation}
    \underset{\substack{v\in W^{1,\infty}(\Omega,\R^d)\\ \lVert \nabla v \rVert_{L^{\infty}} \leq 1}}{\mathrm{arg\;min}}\; J^\prime(\Omega)\, v.
    \label{eq:LipschitzSteepestDescent}
\end{equation}
and then potentially relax the problem.
The descent field is obtained by the solution of a minimization problem for $2 \leq p < \infty$ reading
\begin{equation}
    \underset{v \in W^{1,p}(\Omega,\R^d)}{\mathrm{arg\;min}}\; \frac{1}{p} \int_{\Omega} \lVert \nabla v \rVert_2^p \;\mathrm{d} x + J^\prime(\Omega)\, v
    \label{eq:shapeDescentMinimization}
\end{equation}
Since the classical Hilbert space setting is recovered for the linear case $p=2$, this approach can also be understood as a regularization of such methods.
It is demonstrated that this approach is superior in terms of representation of sharp corners as well as the overall mesh quality and yields improving results for increasing $p$ \cite{mueller2021}.
On the downside, the stated numerical examples also made clear that it is challenging to compute solutions for $p>5$ due to serious difficulties in numerical accuracy and the need to iterate over increasing $p$ with the presented solution technique.\\

\noindent We now assume that the shape derivative $J^\prime(\Omega)$ can be expressed via integrals over the domain and the boundary.
From the perspective of shape calculus \cite{delfour2011, sokolowski1992} this is a reasonable choice to cover relevant applications.
An exemplary computation including geometric constraints can be found in \cite{schulz2016}.
For the finite setting we recover the general formulation of a problem for the $p$-Laplacian $\Delta_p v = \nabla \cdot (\lVert \nabla v \rVert _2^{p-2} \nabla v)$ reading
\begin{equation}
    \underset{u \in \mathcal{U}^p}{\mathrm{arg\;min}}\; J_p(u) = \frac{1}{p} \underbrace{\int_{\Omega} \lVert \nabla (u+g) \rVert_2^p}_{ =:\, \lVert u+g \rVert_{X^p(\Omega)}^p} \;\mathrm{d} x - \int_{\Gamma} h u  \;\mathrm{d}\Gamma - \int_{\Omega} f u \;\mathrm{d} x \quad
    \label{eq:pLaplaceZeroTraceMinimization}
\end{equation}
over the set $\mathcal{U}^p := \{ u \in W^{1,p}(\Omega,\R^d):\, u=0 \text{ a.e. on } \partial\Omega\setminus\Gamma\}$ with $v := u + g$ for an arbitrary prolongation of $g$ to the whole domain in $\{ v \in W^{1,p}(\Omega,\R^d):\, v=g \text{ a.e. on } \partial\Omega\setminus\Gamma\}$.
The corresponding Euler-Lagrange equation is given by
\begin{equation*}
    \left.\begin{array}{ll}
        -\Delta_p v = f &\text{in } \Omega,\\
        \lVert\nabla v \rVert_2^{p-2} \partial_{\eta} v = h & \text{on } \Gamma,\\
        v=g &\text{on } \partial\Omega\setminus\Gamma
    \end{array}\right\}.
    \label{eq:pLaplaceEulerLagrange}
\end{equation*}

\noindent Consequently, in order to obtain an efficient shape optimization algorithm, it is necessary to find a solid routine to solve this problem for a preferably high order $p$. 
On the other hand, it is desirable to directly solve the non-relaxed limit case for $p=\infty$ to obtain analytically valid and possibly superior results.\\

\noindent The remainder is structured as follows: 
In section \ref{sec:HighOrderDescent} we construct the vector-valued extension of the algorithm and integrate support for Neumann boundary conditions. 
We present numerical results obtained with this algorithm in section \ref{sec:NumericalResults}.
After that we discuss we discuss a version of the presented algorithm for $p=\infty$ before we draw conclusions in section \ref{sec:Conclusion}.
\section{High-order \texorpdfstring{$p$}{p}-harmonic descent}
\label{sec:HighOrderDescent}

\noindent In \cite{loisel2020} an algorithm to solve scalar $p$-Laplace problems with Dirichlet boundary conditions is presented in order to show it is solvable in polynomial time.
The approach relies on interior-point methods and the theory of self-concordant barriers, including estimates in terms of the barrier parameter given by Nesterov \cite{nesterov2004}.
Besides the computational complexity, one of the main advantages is that the solution to the linear Laplacian is a sufficient initial guess for any $p$ and no iteration over $p$ is required.
We construct an extension to the algorithm for vector-valued functions featuring mixed Dirichlet and Neumann boundary conditions in order to apply it to shape deformation problems.
However, we will only introduce minor changes to the original proof and show that the polynomial estimate holds with mixed boundary conditions in the scalar setting.\\

\noindent First, we recall some basic notation for finite elements. 
Let $h_\Omega$ be a parameter and $T_{h_\Omega}$ a triangulation of $\Omega$ with $n$ nodes, $m$ elements and quasi-uniformity parameter $\rho_{\Omega}$. 
Further let $V_{h_\Omega}$ denote the space of piece-wise linear Lagrange elements over $T_{h_\Omega}$.
The vector-valued finite element coefficient vector $u \in \R^{nd^\prime}$ is given by a $d'$-block for each node $k$ associated with a basis function $\Phi_k$.
This allows us to extend the notion of discrete derivative matrices to vector-valued functions by $D^{(j,r)} \in \mathbb{R}^{m\times nd^\prime}$ with entries $D_{i,d^\prime(k-1)+r}^{(j,r)} = \frac{\partial}{\partial x_j}\Phi_k(x^{(i)})$ for elements $i=1,\ldots,m$ and nodes $k=1,\ldots,n$ being non-zero if $k \in \mathrm{spt}\; i$. 
Meaning the multiplication to a coefficient vector returns the discrete derivative in direction $j$ of the $r$-th image dimension on each element midpoint $x^{(i)}$.
The vector of weights is given by $\omega^{(l)}$, where $l$ denotes the number of local quadrature points. 
This also means for triangular elements using the mid-point rule $\omega^{(1)} = \omega$ is the vector of element volumes.
The discretization of the $p$-Laplacian term from the problem (\ref{eq:pLaplaceZeroTraceMinimization}) is then given by
\begin{equation*}
    \lVert u+g \rVert_{X^p(\Omega)}^p = 
    \sum_{i=1}^m \omega_i \left( \sum_{j=1}^d \sum_{r=1}^{d^\prime} [D^{(j,r)}(u+g)]_i^2 \right)^\frac{p}{2}.
\end{equation*}

\noindent Further, we define basis matrices $E \in \R^{mld^\prime \times n d^\prime}$ returning the function value for all image dimensions on the $l$ local quadrature points of all elements on multiplication to a coefficient vector. 
This discrete operator will later simplify the proof by allowing us apply similar techniques as those for $D^{(j,r)}$ also for the occurring mass matrices $M=[E^{(l)}]^\intercal W^{(l)} E^{(l)}$ where $W^{(l)} = \mathrm{diag}(\omega^{(l)})$.\\

\noindent Note that all definition hold similarly for boundary elements and will be denoted by a bar, e.g. $\bar{M} = [\bar{E}^{(l)}]^\intercal \bar{W}^{(\bar{l})} \bar{E}^{(l)} \in \R^{nd^\prime \times nd^\prime}$.\\

\noindent With these connections established, we can now state the discretized and reformulated problem for the vector-valued $p$-Laplacian with mixed Dirichlet and Neumann boundary data in the following lemma.

\begin{lemma}
    The problem (\ref{eq:pLaplaceZeroTraceMinimization}) of minimizing $J_p(u)$ over the given finite element space $V_h$ with $1 \leq p < \infty$ satisfying the additional upper bound $\omega_i \lVert \nabla(u+g)\vert_{K_i} \rVert_2^p \leq R$ is equivalent to the classical convex problem
    \begin{equation}
        \min_{x \in \mathcal{Q}_p} \langle c,x \rangle \text{ with } c = 
        \begin{bmatrix}
            -Mf-\bar{M}h\\ 
            \frac{\omega}{p}
        \end{bmatrix}
        \label{eq:ReformulatedFiniteMinimization}
    \end{equation}
    with the constrained search set given by
    \begin{equation}
        \mathcal{Q}_p = \left\{ (u,s) \in \R^n \times \R^m \, : \, s_i \geq \left(\sum_{j=1}^{d} \sum_{r=1}^{d'} [D^{(j,r)}(u+g)]_i^2\right)^{\frac{p}{2}} \land \omega_i s_i \leq R \right\}.
        \label{eq:FiniteConstrainedSet}
    \end{equation}
\end{lemma}

\noindent The obtained problem is now a classical convex optimization problem by the minimization of a scalar product over a constrained set. 
An algorithm is obtained by constructing a self-concordant barrier for $\mathcal{Q}^p$, computing its first and second derivative and then applying an interior-point method \cite{nesterov2004}. 

\begin{lemma}
    \label{lem:FiniteBarrierFunction}
    A $4m$-self-concordant barrier for $\mathcal{Q}_p$ is given by the function
    \begin{equation*}
        \begin{gathered}     
            F(u,s) = -\sum_i \log z_i - \sum_i \log \tau_i \quad\text{where}\\
            z_i = s_i^{2/p} - \sum_{j=1}^{d}\sum_{r=1}^{d^\prime}[(\underbrace{D^{(j,r)} u + D^{(j,r)} g}_{=: y^{(j,r)}})_i]^2 \quad\text{and}\quad \tau_i = R-\omega_i s_i.
        \end{gathered}
    \end{equation*}
\end{lemma}

\begin{remark}
    By construction the barrier function $F$ is twice differentiable with the first derivative reading
    \begin{equation*}
        \begin{aligned}
            F' &= \begin{bmatrix} 
                F_u \\ 
                F_s 
            \end{bmatrix} 
            \text{ where }\\
            F_{u} &= 2 \sum_{j=1}^{d} \sum_{r=1}^{d'} [D^{(j,r)}]^\intercal \frac{y^{(j,r)}}{z} 
            \quad\text{and}\quad
            F_{s} = -\frac{2}{p}\frac{1}{z}s^{2/p-1} + \frac{\omega}{\tau}.
        \end{aligned}
    \end{equation*}
    The second derivative is given by 
    \begin{equation*}
        \begin{aligned}
            F'' &= \begin{bmatrix} F_{uu} & F_{us} \\ F_{us}^\intercal & F_{ss} \end{bmatrix} \text{ where}\\
            F_{uu} &= 2 \sum_{j=1}^{d} \sum_{r=1}^{d'} [D^{(j,r)}]^\intercal Z^{-1} D^{(j,r)}\\
                &\quad+ 4 \sum_{j_1=1}^{d} \sum_{r_1=1}^{d'}\sum_{j_2=1}^{d} \sum_{r_2=1}^{d'} (Y^{(j_1,r_1)}D^{(j_1,r_1)})^\intercal Z^{-1} D^{(j)} (Y^{(j_2,r_2)}D^{(j_2,r_2)}),\\
            F_{us} &= -\frac{4}{p}\sum_{j=1}^{d} \sum_{r=1}^{d'} (Y^{(j,r)}D^{(j,r)})^\intercal Z^{-2} S^{2/p-1},\\
            F_{ss} &= -\frac{2}{p} \left(\frac{2}{p} - 1 \right) Z^{-1} S^{2/p-2} + \frac{4}{p^2} Z^{-2} S^{4/p-2} + W^2 T^{-2},\\
            S &= \diag(s),\, W = \diag(\omega),\, Y = \diag(y),\, Z = \diag(z),\, T = \diag(\tau).
        \end{aligned}
    \end{equation*}
\end{remark}
\begin{remark}
    Note that in the construction the $\frac{p}{2}$-th power of the norm has been moved. 
    Thus the additional constrained $s_i \geq 0$ would be required, as stated in the original version. 
    However, $z_i \rightarrow \infty$ as $s_i \searrow \lVert \nabla(u+g)\vert_{K_i} \rVert_2^2$ or $s_i \nearrow -\lVert \nabla(u+g)\vert_{K_i} \rVert_2^2$ and thus leave us with a correct barrier on the intended set as well additional separated set. Therefore, we drop that condition here and in practice we ensure it by the choice of the initial value with $s_i \geq 0$ and keep track via the line-search in the adaptive path-following. This not only simplifies the notation, but also reduces the computational effort.
\end{remark}

\noindent For completeness, we state the proof for the computational complexity in the scalar case with additional Neumann boundary conditions. 
Note that the obtained bound on the iterations differs from the one in the original version \cite[Theo. 1]{loisel2020}. 
This stems from a change we have to introduce to the Hessian and subsequent computations as well as using a different estimate in terms of the barrier parameter, which features $\lVert c \rVert^*_{x^*_F}$ instead of $\lVert \hat{x} \rVert^*_{x^*_F}$.
However, the estimate on the required iterations for a naive interior-point method is only of theoretical interest. 
Even though the bound is not sharp, the required iterations are not reasonable for practical applications.
Therefore, variations with modified step length \cite{nesterov2001} or an adaptive step size control \cite{loisel2020} are used.
Although estimates are worse in this setting, we will see for the latter in section \ref{sec:NumericalResults} a significantly improved performance.
Consequently, we will not show theoretical results for the vector-valued setting.\\

\begin{theorem}
    Let $1 \leq p <\infty$. Assume that $\Omega \subset \R^{d}$ is a Lipschitz polytype of width L and that $T_{h_\Omega}$ is a quasi-uniform triangulation of $\Omega$, parametrized by $0 < h_\Omega < 1$ and with quasi-uniformity parameter $1 \leq \rho_\Omega < \infty$. 
    Further assume $g \in W^{1,p}(\Omega)$, $f \in \mathrm{L}^q(\Omega)$  and $h \in \mathrm{L}^q(\Gamma)$ with conjugated exponents $\frac{1}{p} + \frac{1}{q} = 1$ are piece-wise linear on $T_{h_\Omega}$ and let $V_{h_{\Omega}} \subset W_{0}^{1,p}(\Omega)$ be the piece-wise linear finite element space on $T_{h_{\Omega}}$ whose trace vanishes.
    Fix a quadrature $Q$ with positive weights such that the integration is exact, $R \geq R^* := 2(1 + \lVert g \rVert^p_{X^p(\Omega)})$ sufficiently large and let $\varepsilon > 0$ be an accuracy.\\
    
    \noindent In exact arithmetic, a naive interior-point method consisting of auxiliary and main path-following using the barrier function from lemma \ref{lem:FiniteBarrierFunction} to minimize $J_p(u)$ over $u \in V_{h}$, starting from $\hat{x} = (0,\hat{s})$ with $\hat{s}_i = 1 + (\sum_j \sum_r [D^{(j,r)}g]_i^2)^{p/2}$, converges to the global minimizer in $V_h$ in at most
    \begin{equation*}
        \begin{aligned}
            N &\leq 14.4 \sqrt{|\Omega|d! h^{-d}} [K^* + \log( \varepsilon^{-1} h_{\Omega}^{-1-7.5d} R^5 (1 + \lVert g \rVert_{X^p(\Omega)})(\lVert f \rVert_{\mathrm{L}^q(\Omega)} + \lVert h \rVert_{\mathrm{L}^q(\Gamma)} + 1))].
        \end{aligned}
    \end{equation*}
    iterations. 
    This results in a computational complexity denoted by $\mathcal{O}(\sqrt{n} \log(n))$.
    The constant $K^{*} = K^{*}(\Omega,\rho_{\Omega}, Q)$ depends on the domain $\Omega$, the quasi-uniformity parameter $\rho_{\Omega}$ of the triangulation and the quadrature $Q$. 
    At convergence, $u$ satisfies
    \begin{equation*}
        J_p(u) \leq \min_{\substack{v\in V_{h_\Omega}\\ \frac{1}{p}\lVert v+g \rVert^p_{X^p} \leq R}} J_p(v) + \varepsilon.
    \end{equation*}
\end{theorem}

\begin{proof}
    \noindent In order to compute the final estimate, we need to find a bound for $\lVert c \rVert_2$.
    Let $f \in \mathrm{L}^q(\Omega)$ and $h \in \mathrm{L}^q(\Gamma)$ piece-wise linear such that the quadrature on $l$ points is exact. 
    We can obtain a bound on a similar way as the bound for $\lVert F'(\hat{x}) \rVert_2$ in the original paper.
    Start with
    \begin{equation*}
        \lVert c \rVert_2 = \left\lVert \begin{bmatrix} -Mf-\bar{M}h \\ \frac{\omega}{p} \end{bmatrix} \right\rVert_2 
        \leq \lVert Mf \rVert_2 + \lVert \bar{M}h \rVert_2 + \frac{1}{p}\lVert \omega \rVert_2.
    \end{equation*}
    
    \noindent In accordance with the original proof we can bound the first term by
    \begin{equation*}
        \lVert Mf \rVert_2 = \lVert [E^{(l)}]^\intercal W^{(l)} E^{(l)} f \rVert_2 \leq \lVert [E^{(l)}]^\intercal W^{(l)} \rVert_2 \lVert E^{(l)}f \rVert_2
    \end{equation*}
    with the same idea of using $\Vert [E^{(l)}]^\intercal W^{(l)} \Vert_2^2 \leq \omega^{(l)}_{\max}\rho ([E^{(l)}]^\intercal W^{(l)} E^{(l)})$ and then bounding the spectral radius $\rho$.
    This is even easier here, because
    \begin{equation*}
        w^\intercal E^\intercal W E w = \int_{\Omega} w^2 dx = \lVert w \rVert^2_{\mathrm{L}^2(\Omega)} \leq [K_{\Omega}']^2 h_{\Omega}^{2d} \lVert w \rVert_2^2
    \end{equation*}
    can be estimated without further inequalities.
    \noindent The remainder can now be bound correspondingly with the equivalence of $p$-norms in finite dimensions
    \begin{equation*}
        \begin{aligned}
            \lVert Ef \rVert_2 &\leq [\omega^{(l)}_{min}]^{-1/2} \left( \sum_{i=1}^{ml} \omega^{(l)}_i [E^{(l)}f]_i^2 \right)^{\frac{1}{2}}\\
            &\leq [\omega^{(l)}_{min}]^{-1/2} (ml)^{1/2} \left( \sum_{i=1}^{ml} \omega^{(l)}_i [E^{(l)}f]_i^q \right)^{\frac{1}{q}}\\
            &\leq (\frac{\omega_{min}}{C_Q})^{-1/2} (ml)^{1/2} \lVert f \rVert_{\mathrm{L}^q(\Omega)}
        \end{aligned}
    \end{equation*}
    where we used that the weights on the reference elements are fixed positive and thus there exists a constant $C_Q > 0$ that only depends on the quadrature $Q$ such that $C_Q^{-1} \omega_{\min} \leq \omega^{(l)}_{\min} \leq \omega_{\min}$.
    Combing the results, we get
    \begin{equation*}
        \begin{aligned}
            \lVert Mf \rVert_2 &\leq \left(K_{\Omega}' h_{\Omega}^d \right) \left(\omega_{min}^{-1/2} C_Q^{1/2} l^{1/2} m^{1/2} \lVert f \rVert_{\mathrm{L}^q(\Omega)} \right)\\
            &\leq K_Q K_{\Omega}'' \lVert f \rVert_{\mathrm{L}^q(\Omega)}.
        \end{aligned}
    \end{equation*}
    
    \noindent A bound for the middle term can be obtained in the same way, just in $(d-1)$ dimensions. Using
    \begin{equation*}
        \vert \Gamma \vert = \sum_{i=1}^{\bar{m}}\bar{\omega}_i \geq \bar{m}\bar{\omega}_{\min} \geq \bar{m}\frac{h_{\Omega}^{d-1}}{(d-1)!} \Leftrightarrow \bar{m} \leq \vert \Gamma \vert (d-1)! h_{\Omega}^{d-1}
    \end{equation*}
    this reads
    \begin{equation*}
        \begin{aligned}
            \lVert \bar{M}h \rVert_2 &\leq \left(K_{\Omega}' h_{\Omega}^{d-1} \right) \left(\bar{\omega}_{min}^{-1/2} (C_{Q}\bar{m}\bar{l})^{1/2} \lVert h \rVert_{\mathrm{L}^q(\Gamma)} \right)\\
            &\leq K_{Q} K_{\Omega}'' \lVert h \rVert_{\mathrm{L}^q(\Gamma)}.
        \end{aligned}
    \end{equation*}

    \noindent Finally, using $\lVert \omega \rVert_2 = \left( \sum_{i=1}^m \omega_i^2 \right)^{1/2} \leq \sqrt{m} \omega_{\max} \leq \sqrt{m}\frac{(\rho h)^d}{d!}$ and $p^{-1} \leq 1$ we are able to obtain a bound in the form
    \begin{equation*}
        \lVert c \rVert_2 \leq C^{\ast}_{\Omega} C^{\ast}_{Q}(\lVert f \rVert_{\mathrm{L}^q(\Omega)} + \lVert h \rVert_{\mathrm{L}^q(\Gamma)} + 1).
    \end{equation*}

    \noindent Together with the previous results from the original proof we can now determine the relevant bounds as: 
    \begin{equation*}
        \begin{aligned}
            \lVert c \rVert^*_{x_F^*} &\leq \lambda_{\min}^{-1} \lVert c \rVert_2 \leq \left( c_{\Omega}' R^{-2} h_{\Omega}^{2d} \right)^{-1} \left( C^{\ast}_{\Omega} C^{\ast}_{Q}(\lVert f \rVert_{\mathrm{L}^q(\Omega)} + \lVert h \rVert_{\mathrm{L}^q(\Gamma)} + 1) \right)\\
            &= [c_{\Omega}']^{-1} C^{\ast}_{\Omega} C^{\ast}_{Q} R^2 h_{\Omega}^{-2d} (\lVert f \rVert_{\mathrm{L}^q(\Omega)} + \lVert h \rVert_{\mathrm{L}^q(\Gamma)} + 1)\\
            \lVert F'(\hat{x}) \rVert^*_{x_F^*} &\leq \lambda_{\min}^{-1} \lVert F'(\hat{x}) \rVert_2 \leq \left( c_{\Omega}' R^{-2} h_{\Omega}^{2d} \right)^{-1} \left( C^*_{\Omega} h_{\Omega}^{-1-1.5d} R (1+\lVert g \rVert_{X^p(\Omega)})\right)\\ 
            &= [c_{\Omega}']^{-1} C_{\Omega}^* R^3 h_{\Omega}^{-1-3.5d} (1 + \lVert g \rVert_{X^p(\Omega)})
        \end{aligned}
    \end{equation*}
    Plugging these into the actual bound by Nesterov \cite[Sec. 4.2.5]{nesterov2004}
    \begin{equation*}
        N \leq 7.2 \sqrt{\nu}\, [2\, \ln(\nu) + \ln(\lVert F'(\hat{x}) \rVert^*_{x_F^*}) + \ln(\lVert c \rVert^*_{x_F^*}) + \ln(1/\varepsilon)],
    \end{equation*}
    we obtain the bound in terms of number of grid points as
    \begin{equation*}
        \begin{aligned}
            N &\leq 7.2 \sqrt{4m}[2\log(4m) + \log(1/\varepsilon)\\
                &\quad + \log([c_{\Omega}']^{-1} C_{\Omega}^{\ast} R^3 h_{\Omega}^{-1-3.5d} (1 + \lVert g \rVert_{X^p(\Omega)}))\\
                &\quad + \log( [c_{\Omega}']^{-1} C^{\ast}_{\Omega} C^{\ast}_{Q} R^2 h_{\Omega}^{-2d} (\lVert f \rVert_{\mathrm{L}^q(\Omega)} + \lVert h \rVert_{\mathrm{L}^q(\Gamma)} + 1))]\\
            &\leq 14.4 \sqrt{|\Omega|d! h_{\Omega}^{-d}} [\log([c_{\Omega}']^{-2} [C_{\Omega}^{\ast}]^2 C_{Q}^{\ast} 16 |\Omega|^2 (d!)^2)]\\
                &\quad + \log(\varepsilon^{-1} h_{\Omega}^{-1-7.5d} R^5 (1 + \lVert g \rVert_{X^p(\Omega)})(\lVert f \rVert_{\mathrm{L}^q(\Omega)} + \lVert h \rVert_{\mathrm{L}^q(\Gamma)} + 1) )]\\
            &= 14.4 \sqrt{|\Omega|d! h_{\Omega}^{-d}} [K^*(\Omega, \rho_{\Omega}, Q)\\
                &\quad + \log(\varepsilon^{-1} h_{\Omega}^{-1-7.5d} R^5 (1 + \lVert g \rVert_{X^p(\Omega)})(\lVert f \rVert_{\mathrm{L}^q(\Omega)} + \lVert h \rVert_{\mathrm{L}^q(\Gamma)} + 1) )].
        \end{aligned}
    \end{equation*}
\end{proof}

\begin{remark}
    To find the global minimum, $R$ has to be chosen sufficiently large, so that the solution is contained in the search set $\mathcal{Q}$. For $2 \leq p < \infty$ such an upper bound exists and is given by
    \begin{equation*}
        R = 2 + 4 \lVert g \rVert_{X^p(\Omega)}^p + 8 (p-1) \left[ L^q \lVert f \rVert_{\mathrm{L}^q(\Omega)}^q + C_T^q (L^q + 1) \lVert h \rVert_{\mathrm{L}^q(\Gamma)}^q\right]
    \end{equation*}
    where $C_T$ denotes the Sobolev trace constant \cite{evans2015}, which only depends on $p$ and $\Omega$. While there even exists an uniform upper bound for it\cite{bonder2003}, usually neither of them is computable. Therefore, in the actual implementation we choose a heuristic approach. Start by dropping the term resulting from the Neumann boundary condition and if the values during the iteration come close to the bound, restart with an increased version.
\end{remark}
\begin{proof}
    We follow the original proof for the pure Dirichlet case, showing an upper bound for a minimizing sequence $u_k$ using Hölder's inequality, the modified Friedrichs inequality for $\lVert \cdot \rVert_{X^p(\Omega)}$ and the Sobolev trace theorem for the new term.\\
    
    \noindent Start by assuming $\lVert u \rVert_{X^p(\Omega)} \geq \lVert g \rVert_{X^p(\Omega)}$, since otherwise the bound is trivial, and compute
    \begin{equation*}
        \begin{aligned}
            J(u) &= \frac{1}{p}\lVert u+g \rVert_{X^p(\Omega)}^p - \int_\Omega fu \;\mathrm{d}x - \int_\Gamma hu \;\mathrm{d}\Gamma\\ 
            &\geq \frac{1}{p} (\lVert u \rVert_{X^p(\Omega)} - \lVert g \rVert_{X^p(\Omega)})^p - \lVert f \rVert_{\mathrm{L}^q(\Omega)}\lVert u \rVert_{\mathrm{L}^p(\Omega)} - \lVert h \rVert_{\mathrm{L}^q(\Gamma)}\lVert u \rVert_{\mathrm{L}^p(\Gamma)}\\
            &\geq \frac{1}{p} \lVert u \rVert_{X^p(\Omega)}^p - \frac{1}{p} \lVert g \rVert_{X^p(\Omega)}^p - \lVert f \rVert_{\mathrm{L}^q(\Omega)} L p^{-\frac{1}{p}}\lVert u \rVert_{X^p(\Omega)}^p\\
                &\quad- \lVert h \rVert_{\mathrm{L}^q(\Gamma)} C_T (L p^{-\frac{1}{p}}+1) \lVert u \rVert_{X^p(\Omega)}^p
        \end{aligned}
    \end{equation*}
    \noindent Now we can modify the application of Young's inequality for two forcing terms to $a_1 = 4^{\frac{1}{p}} L p^{-\frac{1}{p}} \lVert f \rVert_{\mathrm{L}^q(\Omega)} $, $a_2 = 4^{\frac{1}{p}} C_T (L p^{-\frac{1}{p}}+1) \lVert h \rVert_{\mathrm{L}^q(\Gamma)} $ and $b_{1,2} = 4^{-\frac{1}{p}} \lVert u \rVert_{X^p(\Omega)}$ and thus get
    
    \begin{equation*}
        \begin{aligned}
            J(u) - J(0) &\geq \frac{1}{2p}\lVert u \rVert_{X^p(\Omega)}^p - \frac{2}{p} \lVert g \rVert_{X^p(\Omega)}^p -\frac{1}{q} \underbrace{4^{\frac{q}{p}}}_{\underset{\text{for}\, p\geq2}{\leq 4}} \underbrace{p^{-\frac{q}{p}}}_{\leq 1} L^q \lVert f \rVert_{\mathrm{L}^q(\Omega)}^q\\
                &\quad- \frac{1}{q} 4^{\frac{q}{p}} C_T^q \underbrace{(L p^{-\frac{1}{p}}+1)}_{\leq (L^q p^{-\frac{q}{p}}+1)} \lVert h \rVert_{\mathrm{L}^q(\Gamma)}^q.
        \end{aligned}
    \end{equation*}
    
    \noindent Therefore, if
    \begin{equation*}
        \lVert u \rVert_{X^p(\Omega)}^p > 4 \lVert g \rVert_{X^p(\Omega)}^p + 8 (p-1) \left[ L^q \lVert f \rVert_{\mathrm{L}^q(\Omega)}^q + C_T^q (L^q + 1) \lVert h \rVert_{\mathrm{L}^q(\Gamma)}^q\right] =: \tilde{R},
    \end{equation*}
    then $J(u) - J(0) \geq 0$. 
    By contradiction any minimizing sequence $u_k$ must fulfill $\lVert u_k \rVert_{X^p(\Omega)}^p \leq \tilde{R}$ for some $k$ large enough.
    Thus, it is in the set $\mathcal{Q}$ and the final bound is obtained by $R := 2(\tilde{R}+1)$ to ensure the construction condition for the initial value.
\end{proof}
\section{Numerical results}
\label{sec:NumericalResults}

In this section, we present results for numerical experiments with the scheme presented in section \ref{sec:HighOrderDescent}.
The data was calculated with an implementation in julia based on the finite element library MinFEM \cite{minfem2020} and visualized with Paraview.
We choose julia since it allows straightforward computations using matrix or vector based operations, easy adaptation of the code and provides great accessibility to accuracy parameters.
Therefore, it is highly suitable for the analysis of algorithms and the experiments performed here.\\

\begin{figure}[!htp]
    \centering
    \begin{subfigure}[b]{0.3\textwidth}
        \centering
        \includegraphics[width=\textwidth]{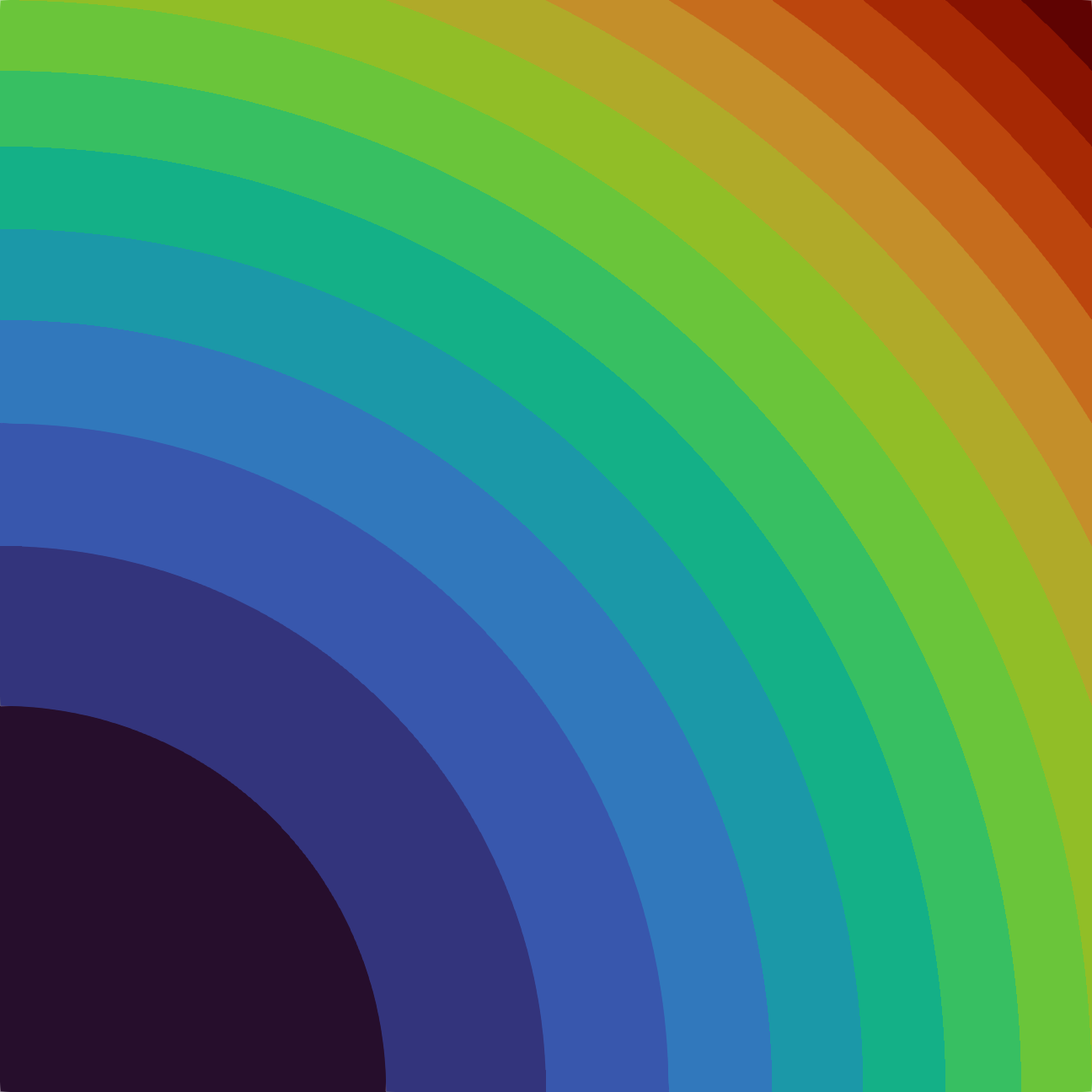}
    \end{subfigure}
    \begin{subfigure}[b]{0.3\textwidth}
        \centering
        \includegraphics[width=\textwidth]{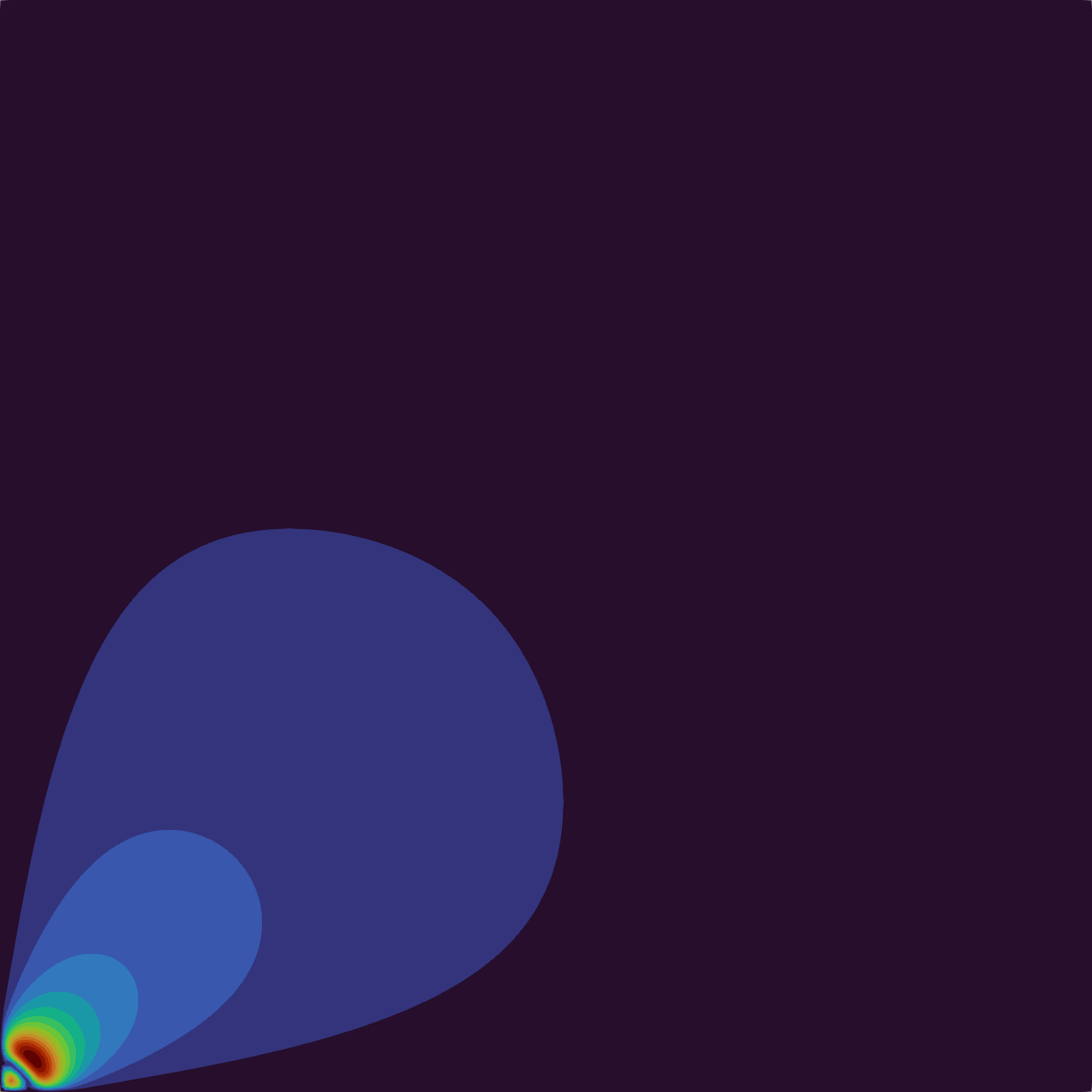}
    \end{subfigure}
    \begin{subfigure}[b]{0.3\textwidth}
        \centering
        \includegraphics[width=\textwidth]{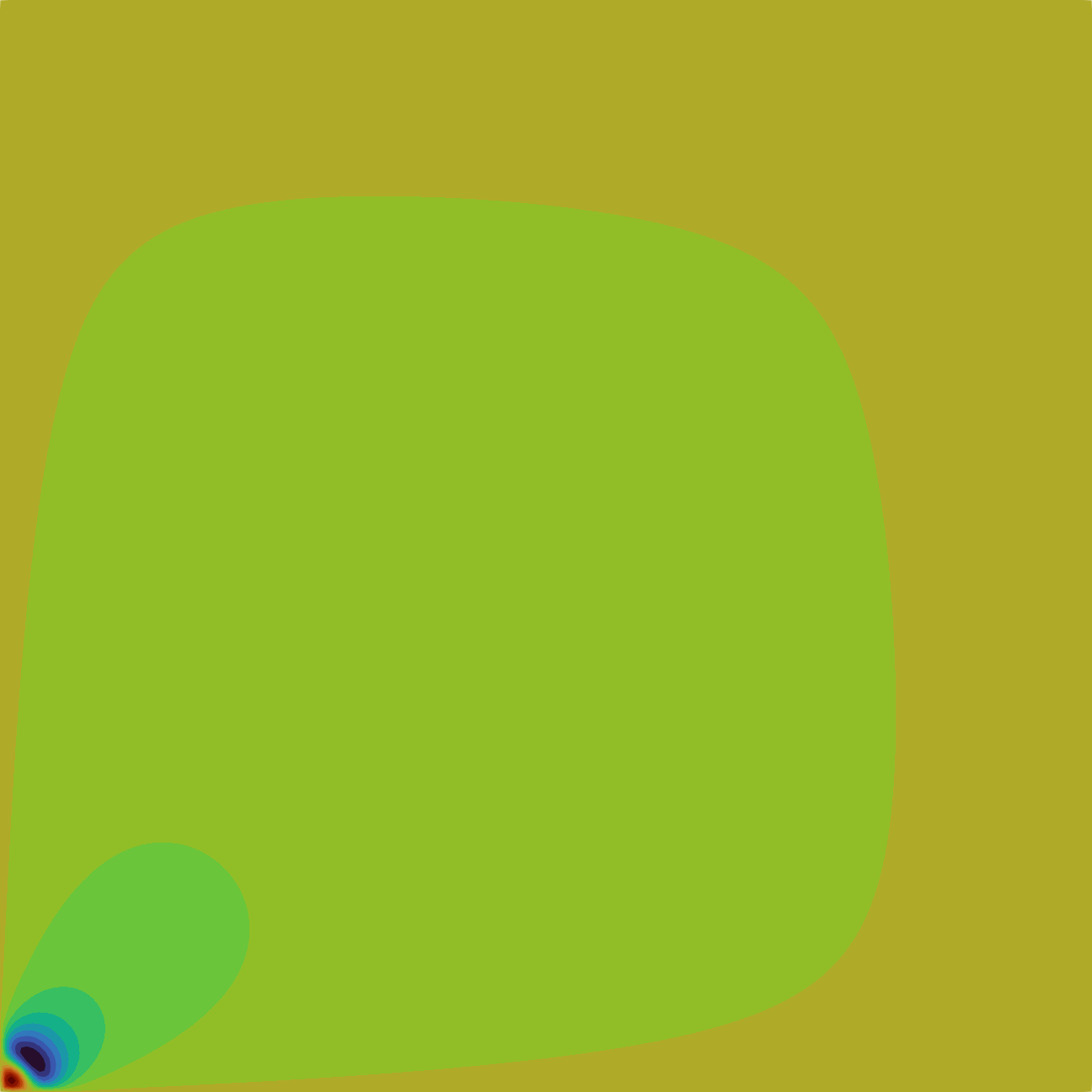}
    \end{subfigure}
    \begin{subfigure}[b]{0.3\textwidth}
        \centering
        \includegraphics[width=\textwidth]{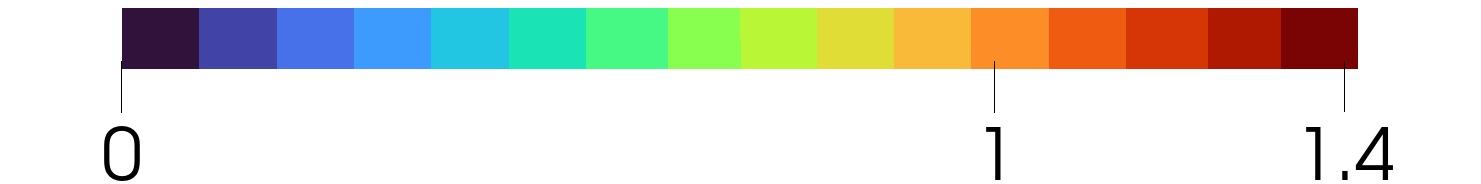}
        \caption{$\lVert v \rVert_2$}
        \label{fig:ValidationSolMagnitude}
    \end{subfigure}
    \begin{subfigure}[b]{0.3\textwidth}
        \centering
        \includegraphics[width=\textwidth]{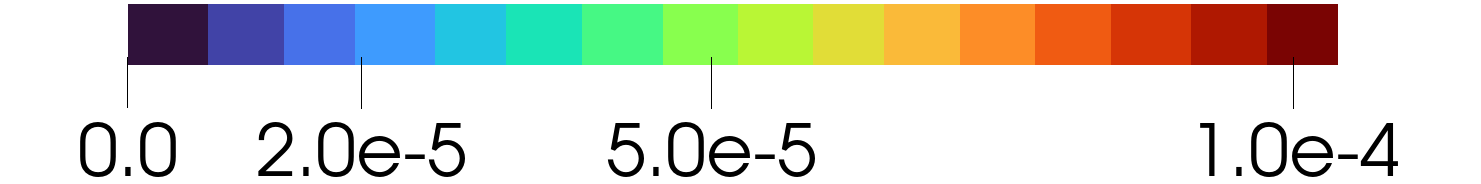}
        \caption{$\lVert v-v^* \rVert_2$}
        \label{fig:ValidationErrorMagnitude}
    \end{subfigure}
    \begin{subfigure}[b]{0.3\textwidth}
        \centering
        \includegraphics[width=\textwidth]{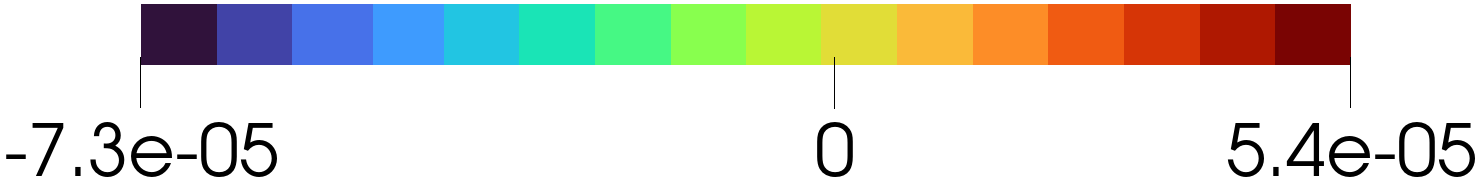}
        \caption{$v_r - v_r^*$}
        \label{fig:ValidationErrorX}
    \end{subfigure}
    \captionsetup{format=hang}
    \caption{Solution and error for validation by method of manufactured solutions for $v^* = \frac{1}{2}\lVert x \rVert_2 \cdot [1,1]^\intercal$ on $\Omega=[0,1]^2$.}
    \label{fig:Validation}
\end{figure}

\noindent We start by validating the algorithm and the implementation. 
For that purpose, we use the method of manufactured solutions \cite{salari2000} to approximate the analytical solution $v^* = \frac{1}{2}\lVert x \rVert_2 \cdot [1,1]^\intercal$ given for the problem
\begin{equation*}
    \left.\begin{array}{ll}
    -\Delta_p v  = -p\, 2^{\frac{p-2}{2}} \lVert x \rVert_2^{p-2} \cdot \begin{bmatrix}1\\ 1\end{bmatrix} & \text{in } \Omega\\
    v = \frac{1}{2}\lVert x \rVert_2^2 \cdot \begin{bmatrix}1\\ 1\end{bmatrix} & \text{on } \partial\Omega
    \end{array}\right\}.
\end{equation*}

\noindent We only test the vector-valued setting, since it contains the scalar case, which has not been validated so far, per component.
Figure \ref{fig:Validation} shows the obtained solution and the error on the unit square discretized by a regular mesh with $40000$ nodes.
The error in the two components is identical and in the order of the intended accuracy $\varepsilon = 10^{-6}$.
The largest errors naturally occurs in the bottom left corner, where a function value close to $0$ has to be approximated.
Thus we consider the results obtained by our implementation as valid.\\

\begin{figure}[!htp]
    \centering
    \includegraphics[width=0.4\textwidth]{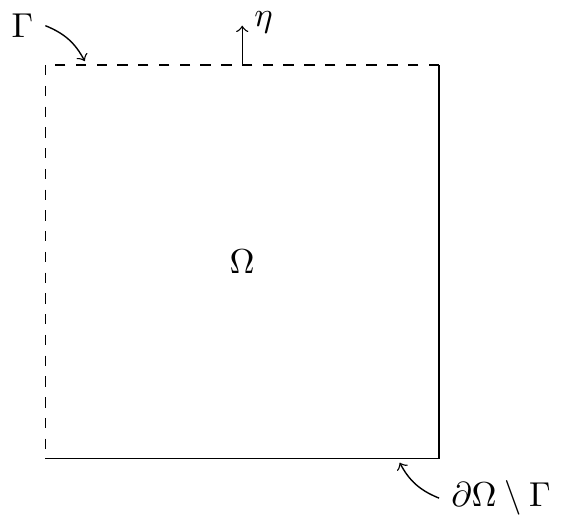}
    \caption{Sketch of domain for exemplary problem.}
    \label{fig:DomainSketch}
\end{figure}

\noindent From now on, we focus on the additional Neumann boundary conditions.
Therefore, we consider the domain $\Omega$ as the unit square, where left and upper boundary are free for deformation and denoted by $\Gamma$.
A sketch of this domain can be found in figure \ref{fig:DomainSketch}.
On the free boundary, we will work with combinations of the function $\hat{h}(x) = \sin(2\pi x_1) - \sin(2\pi x_2)$.
This is a sine wave cycle on each part of the boundary, where the one on the upper boundary is inverted such that the two positive parts are next to the upper left corner.
As we will see later, this construction leads towards a distance function on the boundary of the limit solution instead of multiple hats.\\

\begin{table}[!hbp]
    \centering
    \begin{tabular}{ c || c | c | c || c | c | c || c | c | c ||}
         $h=\phantom{22}$ & \multicolumn{3}{c||}{$\hat{h}(x)$} & \multicolumn{3}{c||}{$\hat{h}(x) \cdot \eta$} & \multicolumn{3}{c||}{$\hat{h}(x) \cdot [1,1]^\intercal$}\\
         $n=\phantom{22}$ & \phantom{~}2500 & 10000 & 40000 & \phantom{~}2500 & 10000 & 40000 & \phantom{~}2500 & 10000 & 40000\\
        \hline
        $p=\phantom{2}2$ & 95 & 102 & 116 & 94 & 101 & 115 & 95 & 104 & 116\\
        $p=\phantom{2}3$ & 95 & 104 & 114 & 93 & 101 & 114 & 96 & 107 & 115\\
        $p=\phantom{2}5$ & 92 & 103 & 113 & 92 & 102 & 107 & 98 & 129 & 116\\
        $p=\phantom{2}8$ & 118 & 213 & 198 & 86 & 99 & 107 & 186 & 204 & 213\\
        $p=15$ & 191 & 253 & 296 & 111 & 126 & 148 & 228 & 278 & 326\\
        $p=25$ & 233 & 308 & 464 & 152 & 204 & 181 & 280 & 361 & 481\\
    \end{tabular}
    \captionsetup{format=hang}
    \caption{Required Newton iterations for solving problems for different boundary source terms with $\hat{h}(x) = \sin(2\pi x_1) - \sin(2\pi x_2)$, number of grid points and PDE parameters $p$.}
    \label{tab:FiniteEffort}
\end{table}

\noindent For this setting, we can observe the number of required Newton iterations in the path-following with adaptive step-size control for various PDE parameter $p$ and refinements the grid.
Table \ref{tab:FiniteEffort} shows these values for the scalar setting and two different prolongations of $\hat{h}(x)$ to a vector-valued setting.
Note that the vector-valued problems can be a significantly different problem.
For example applying $\hat{h}(x)$ to the outer normal vector results in a problem, where per component only one of the free edges features a sine wave and the other one is homogeneous.
This is component wise a simpler problem than the regular scalar one, which is despite the connection of the components visible in the required iterations.
The first observation is that much higher orders $p$ are possible when the boundary features a free part.
For the pure Dirichlet setting the numerical maximal value was around $p=5$, here the source term reduces the stiffness of the problem such that even $p=25$ is possible.
In general, we see for all settings the expected behaviour of increasing iterations for increasing $p$ and $n$ with some exceptions due to the adaptive stepping.
Further, the overall iterations stay comparably small to the number of grid points and thus significantly better than the theoretical estimate.
In the scalar Dirichlet case this behaviour for the adaptive stepping was already known, but it was unclear if it transfers to the vector-valued setting, especially since the analytical problem is inherently more difficult due to nature of the Frobenius norm in the operator.
For the intended application to shape optimization this is a crucial observation, since the computation many deformation fields is required and thus determines the overall runtime.\\

\begin{figure}[!htp]
    \centering
    \begin{subfigure}[b]{0.45\textwidth}
        \centering
        \includegraphics[width=\textwidth]{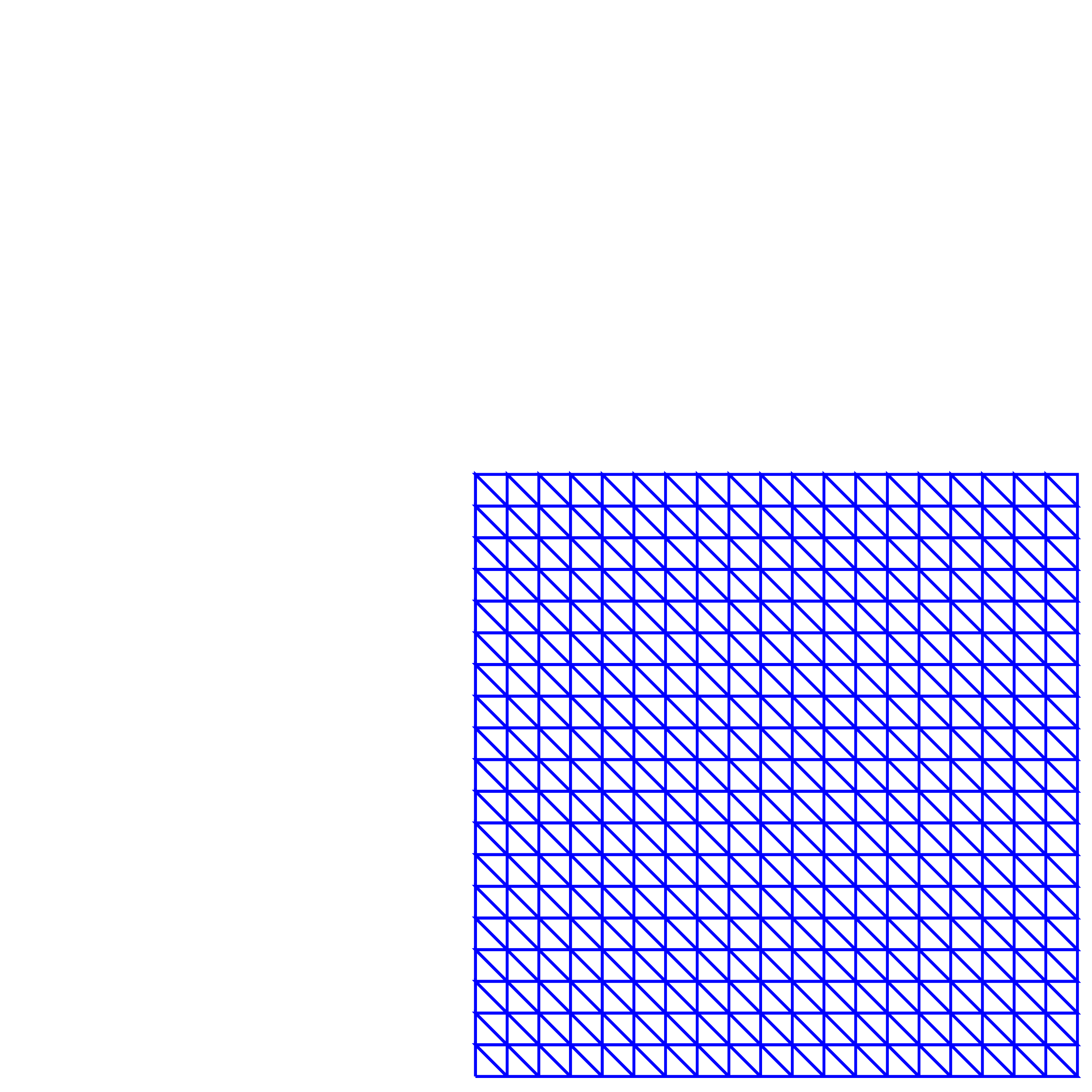}
        \caption{Initial mesh}
    \end{subfigure}
    \begin{subfigure}[b]{0.45\textwidth}
        \centering
        \includegraphics[width=\textwidth]{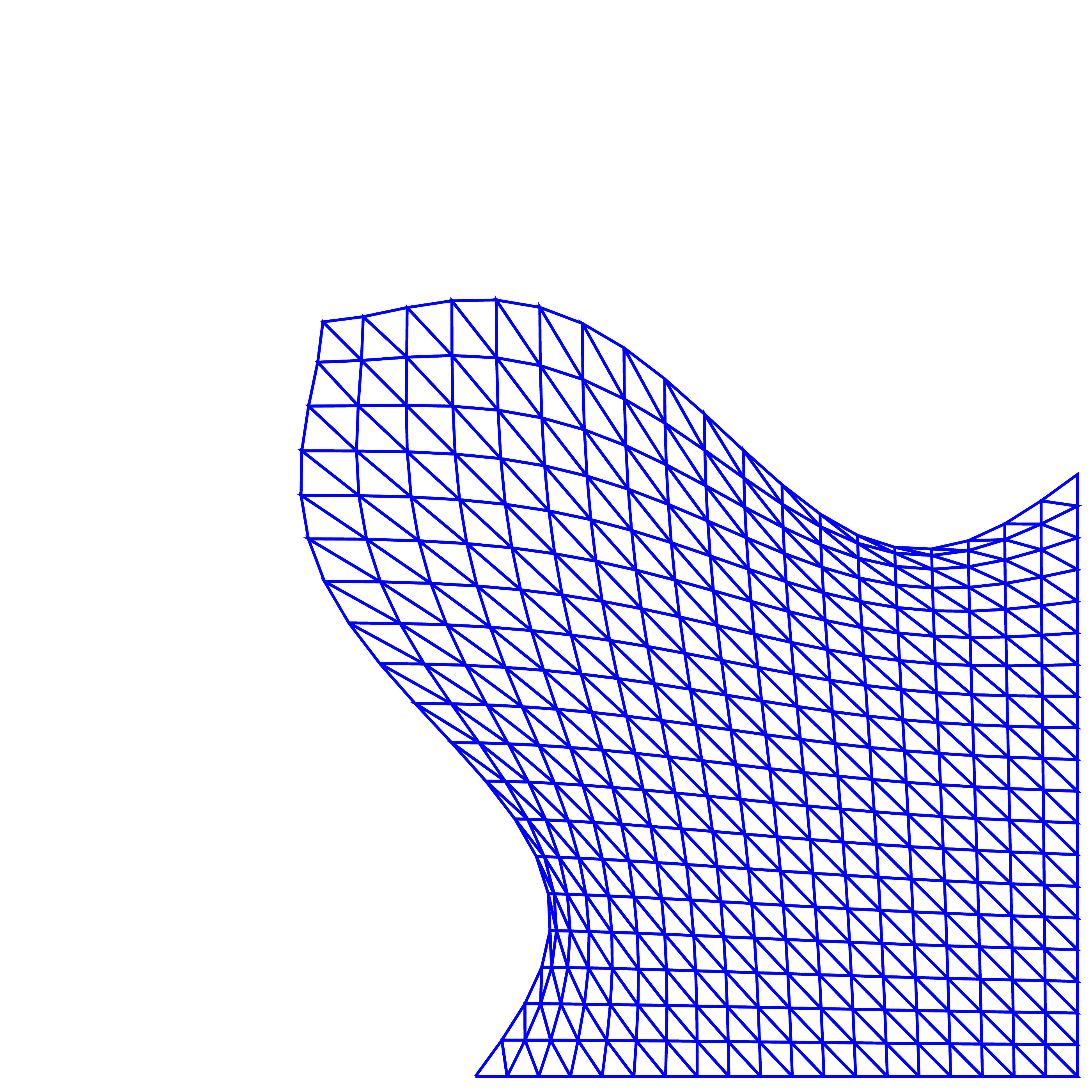}
        \caption{$p=2$}
    \end{subfigure}
    \begin{subfigure}[b]{0.45\textwidth}
        \centering
        \includegraphics[width=\textwidth]{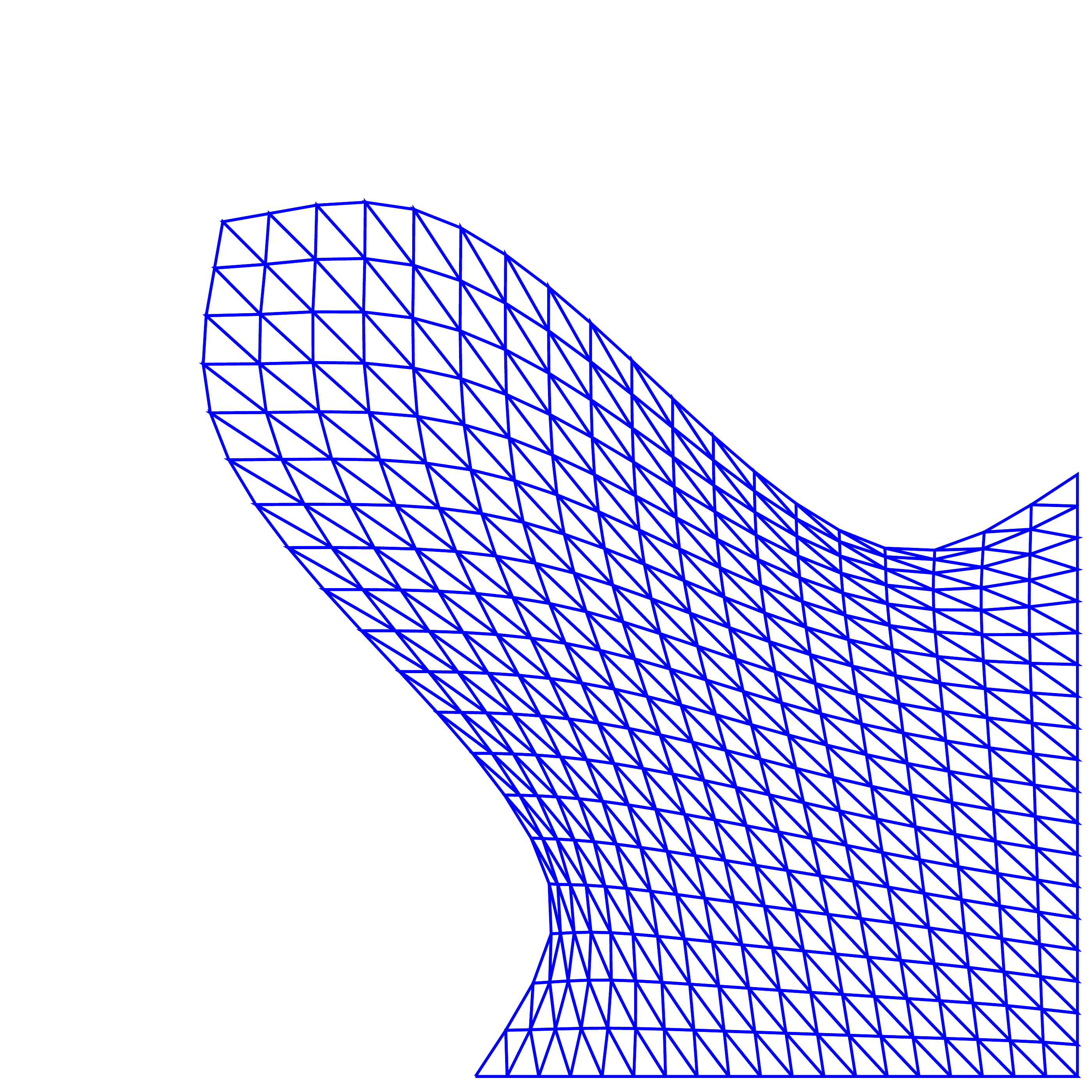}
        \caption{$p=5$}
    \end{subfigure}
    \begin{subfigure}[b]{0.45\textwidth}
        \centering
        \includegraphics[width=\textwidth]{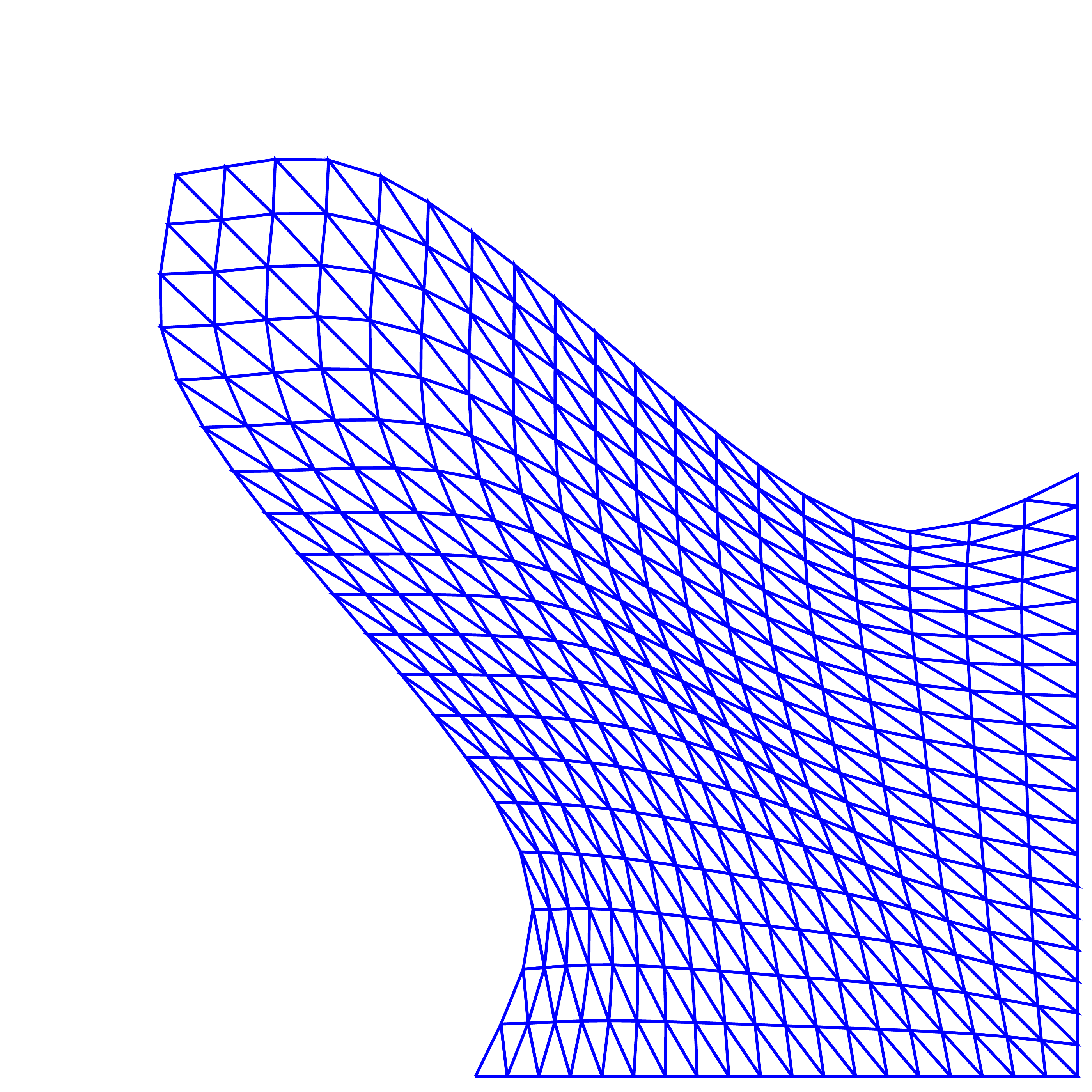}
        \caption{$p=15$}
    \end{subfigure}
    \captionsetup{format=hang}
    \caption{Deformation of a square domain with $h(x) = (\sin(2\pi x_1) - \sin(2\pi x_2)) \cdot \eta$ prescribed on the left and upper boundary for increasing $p$. For comparison the solution for the smaller $p$ is given in blue and the higher order one in red.}
    \label{fig:FiniteDeformations}
\end{figure}

\noindent Now we will use results obtained for $\hat{h}(x) \cdot \eta$ to calculate the deformation of the square by perturbation of identity (\ref{eq:PerturbationOfIdentity}).
A selected sequence of transformed domains are shown in figure \ref{fig:FiniteDeformations} with a reduced number of grid points for improved visibility.
For comparison we take the domain for $p=2$.
This features strong bends near the two fixed endpoints of the free boundary, where the mesh is highly deteriorated.
Increasing the order to $p=5$, the magnitude of the deformation increases significantly and the bends reduce, however the mesh quality is still poor.
When changing to $p=15$, the magnitude changes only slightly as well as the bend.
However, the mesh at the bends is not deteriorated anymore and the elements in the interior are deformed more uniformly.

\section{Descent in \texorpdfstring{$W^{1,\infty}$}{W 1,inf}}
\label{sec:InfiniteDescent}

\noindent In this section, we consider directly solving the steepest descent problem (\ref{eq:LipschitzSteepestDescent}), commonly associated with the variational problem for the $\infty$-Laplacian and discuss the challenges arising.
The first important property of the $\infty$-Laplacian is that the solutions are in general non-unique. 
However, by the approach of regularization with the $p$-Laplacian, we want the limit of those so-called $p$-Extensions, which is known as the \textit{absolutely minimizing Lipschitz extension (AMLE)} due to early work by Aronsson \cite{aronsson2004}.
Here, the term ``absolutely minimizing'' denotes functions $v \in C(\Omega)$ with Lipschitz constant $L_v(V) = L_v(\partial V)$ for all $V \Subset \Omega$.\\

\begin{figure}[!htpb]
    \centering
    \begin{subfigure}[b]{0.45\textwidth}
        \centering
        \includegraphics[width=\textwidth]{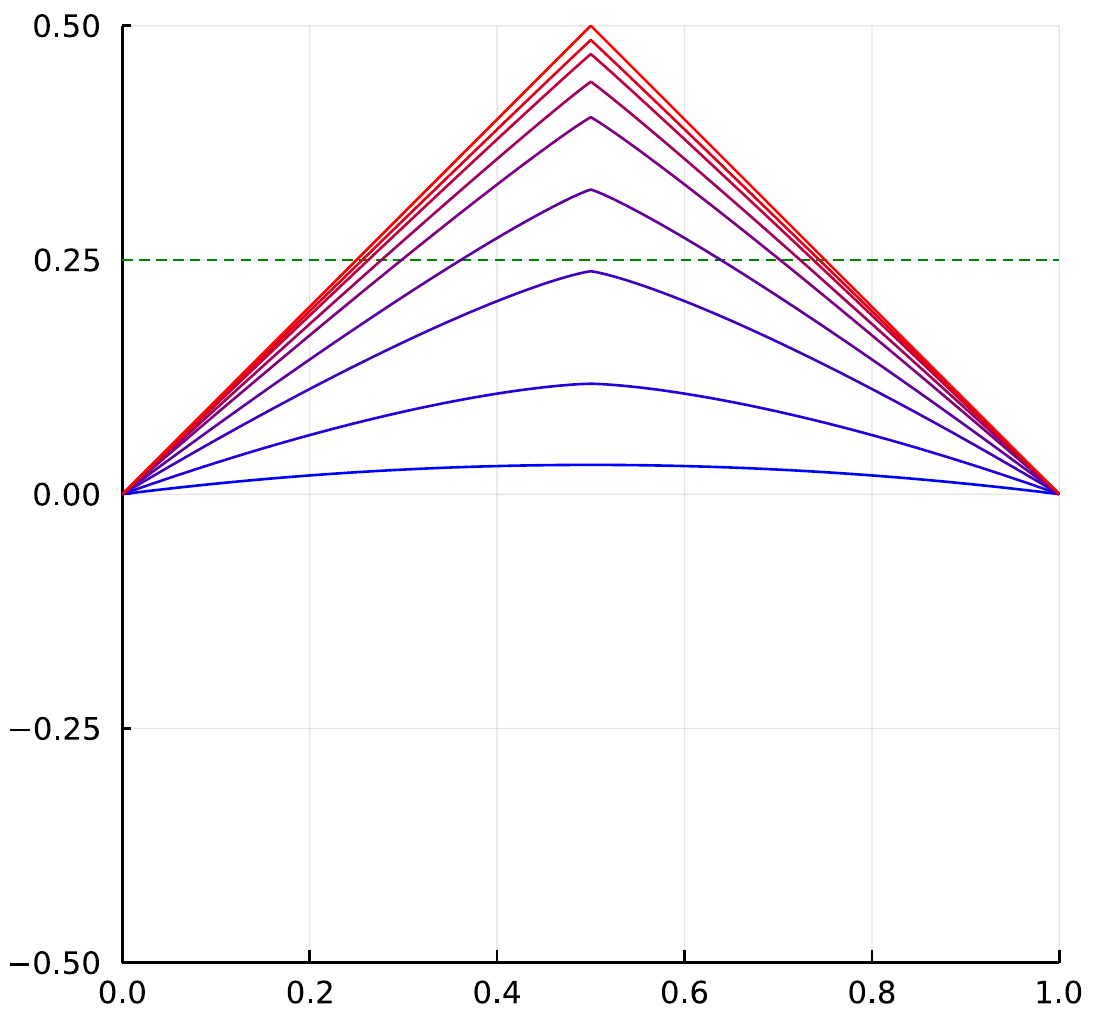}
    \end{subfigure}
    \begin{subfigure}[b]{0.45\textwidth}
        \centering
        \includegraphics[width=\textwidth]{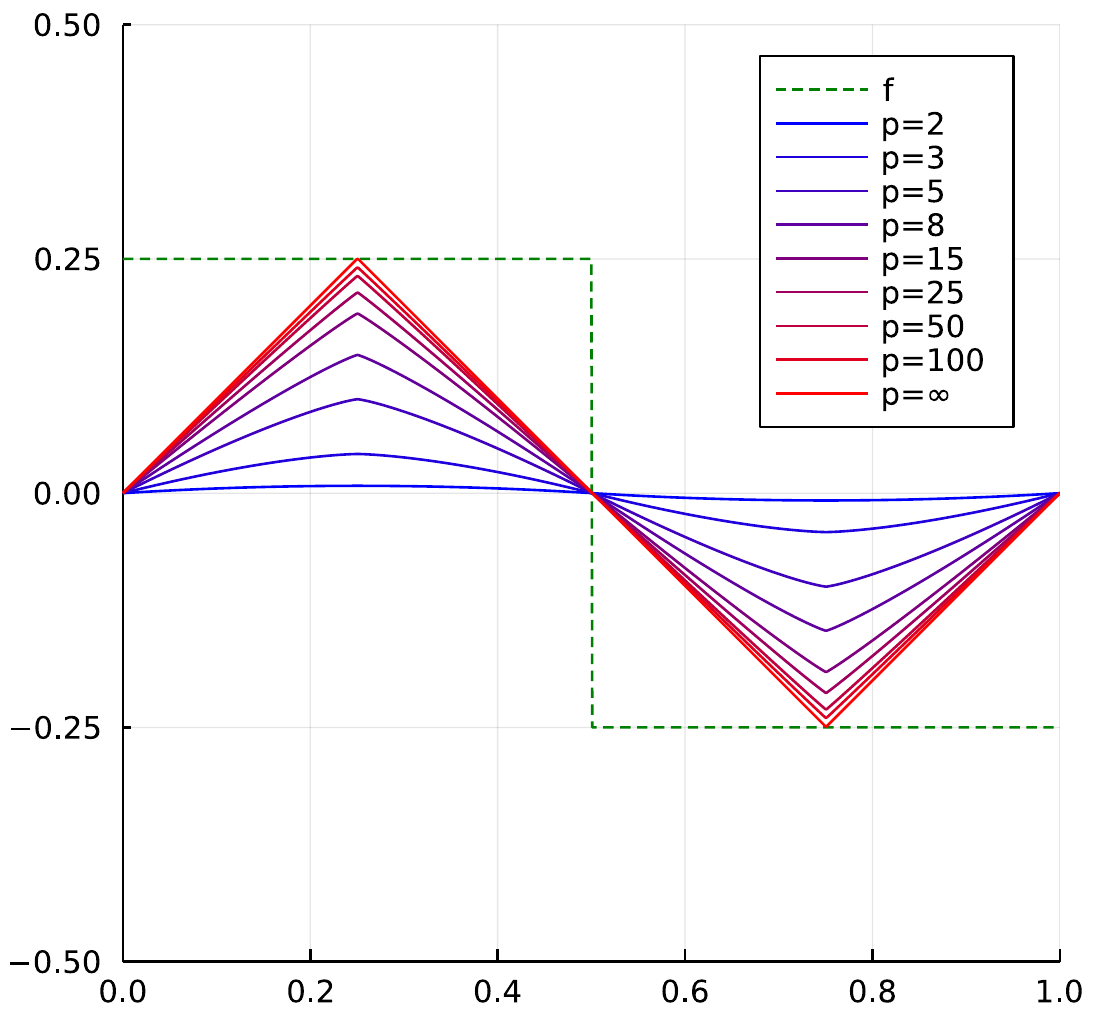}
    \end{subfigure}
    \begin{subfigure}[b]{0.45\textwidth}
        \centering
        \includegraphics[width=\textwidth]{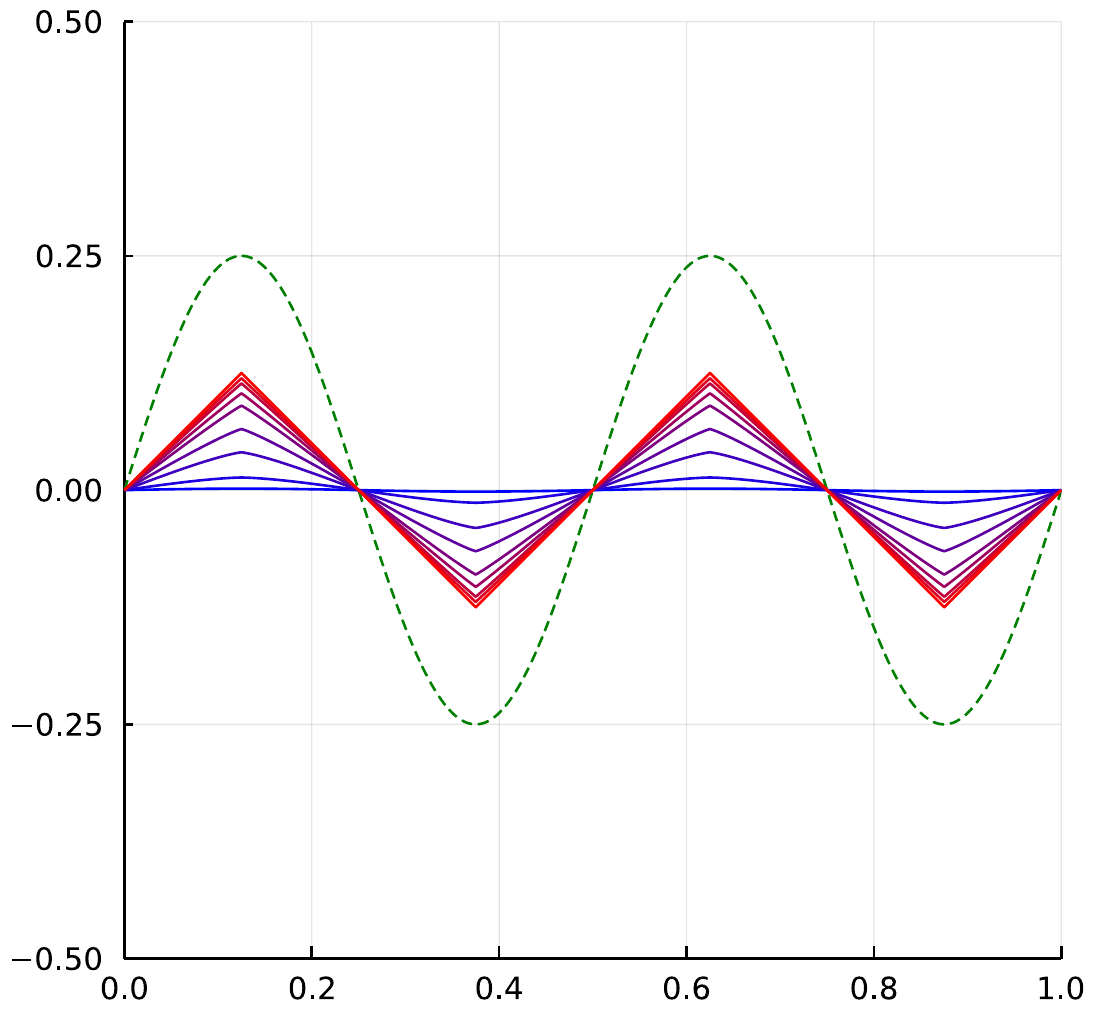}
    \end{subfigure}
    \begin{subfigure}[b]{0.45\textwidth}
        \centering
        \includegraphics[width=\textwidth]{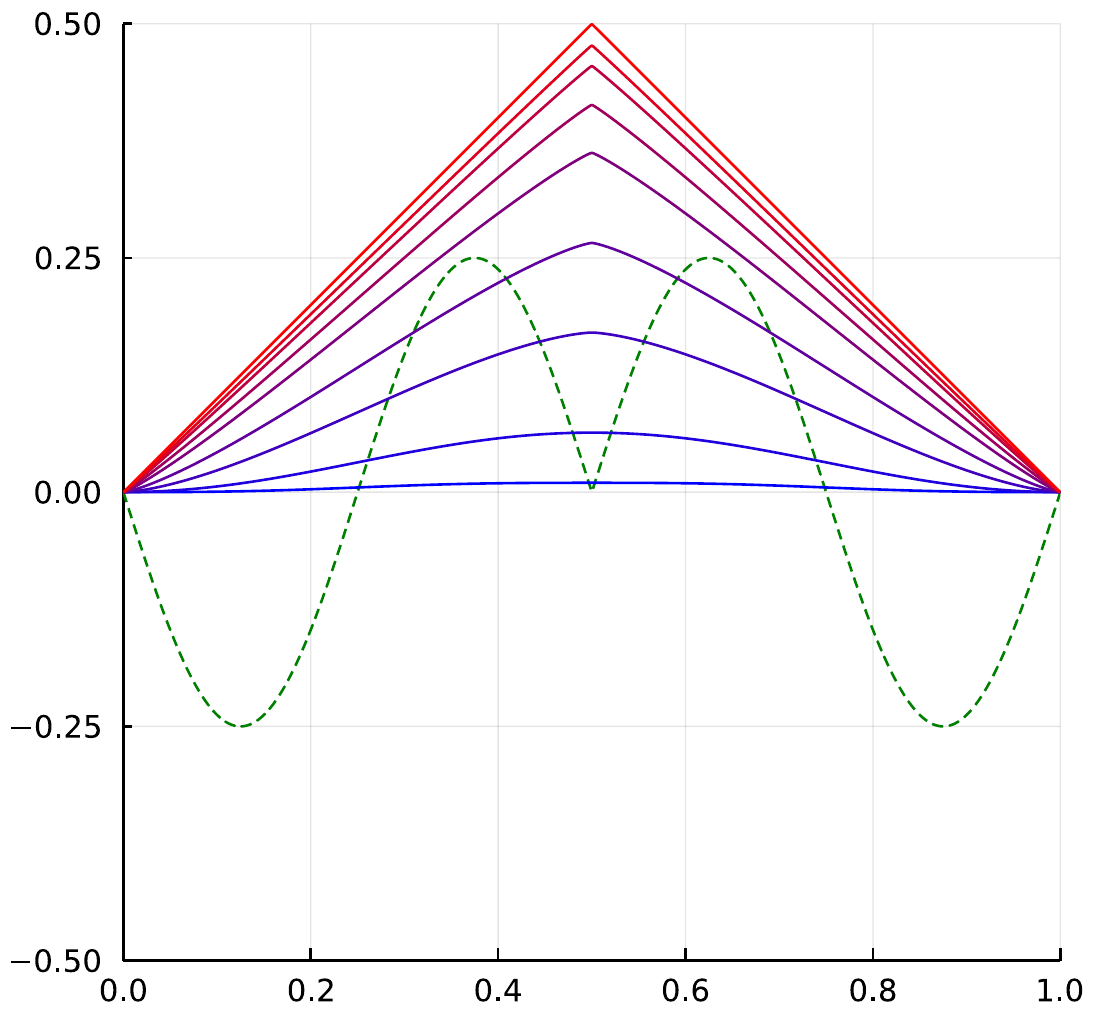}
    \end{subfigure}
    \captionsetup{format=hang}
    \caption{Convergence of numerical solutions for the $p$-Laplacian with increasing order and the analytical limit $p=\infty$ in $[0,1]$ for $g=0$ and different source terms $f$.}
    \label{fig:1DPlot}
\end{figure}

\noindent While the minimization formulation for the problem is common, there is no variational formulation with test functions under the integral \cite{lindqvist2019}.
Further, even only for zero forcing a reformulation to an Euler-Lagrange equation is possible.
With this approach it was shown that unique solutions in the above sense exist for boundary extension problems \cite{jensen1993}.
This remains true for homogeneous Dirichlet problems with non-negative source terms \cite{ishii2005}.
Here, the unique limit solution can be split into two regions.
In the support of the source term, the solution is given by the distance to the boundary $\mathrm{dist}(x,\partial\Omega)$.
The other region is then given by the unique solution of the $\infty$-Laplacian without forcing and the distance as Dirichlet boundary condition.\\

\noindent In one dimension, the situation is simplified and analytical solutions can be computed even for source terms with changing sign. 
Figure \ref{fig:1DPlot} shows such solutions and corresponding solutions of $p$-harmonic relaxations for different source terms.
We can observe that the limit solution always has slope 1.
Its magnitude does not depend on the magnitude of the source term, but only on the length of the interval between two sign changes.
However, the $p$-harmonic solutions depend on it, meaning that the approximation quality depends on the magnitude of the source term as well. 
Another interesting observation can be made in the lower right plot.
The solutions are able to eliminate certain areas with different signs.
While the the limit solution is identical to the one with constant source term, the $p$-harmonic approximations are different.
Especially one can still an impact of the sign change close to the boundary and the kink at tip occurs later.
Further, this one dimensional source term is similar to a normalized version of the Neumann boundary condition or boundary source term used in figure \ref{fig:FiniteDeformations}.
Thus, we would expect a solution with slope 1 everywhere on the boundary for the limit deformation.\\

\begin{figure}[!htbp]
    \centering
    \begin{subfigure}[b]{0.3\textwidth}
        \centering
        \includegraphics[width=\textwidth]{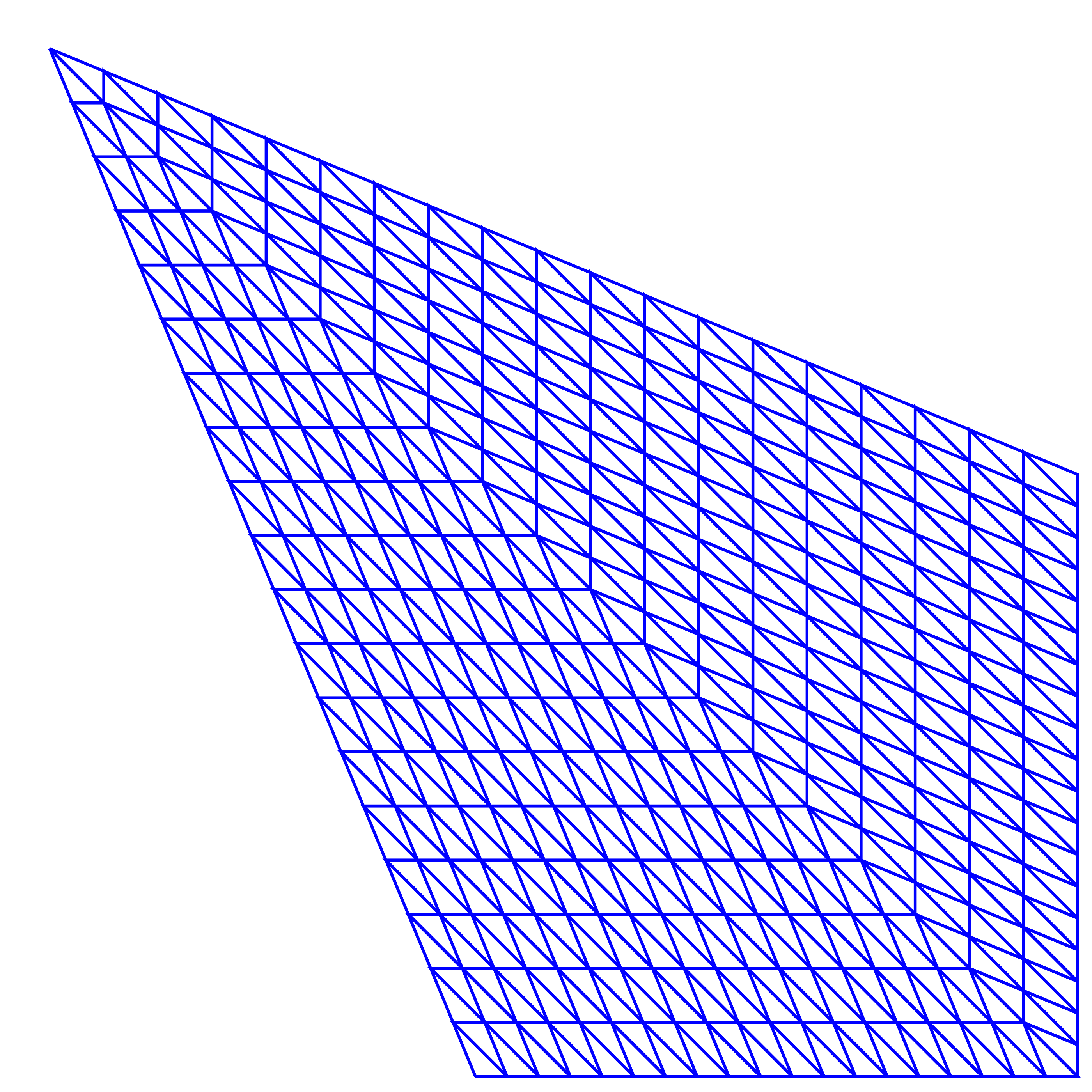}
    \end{subfigure}
    \begin{subfigure}[b]{0.3\textwidth}
        \centering
        \includegraphics[width=\textwidth]{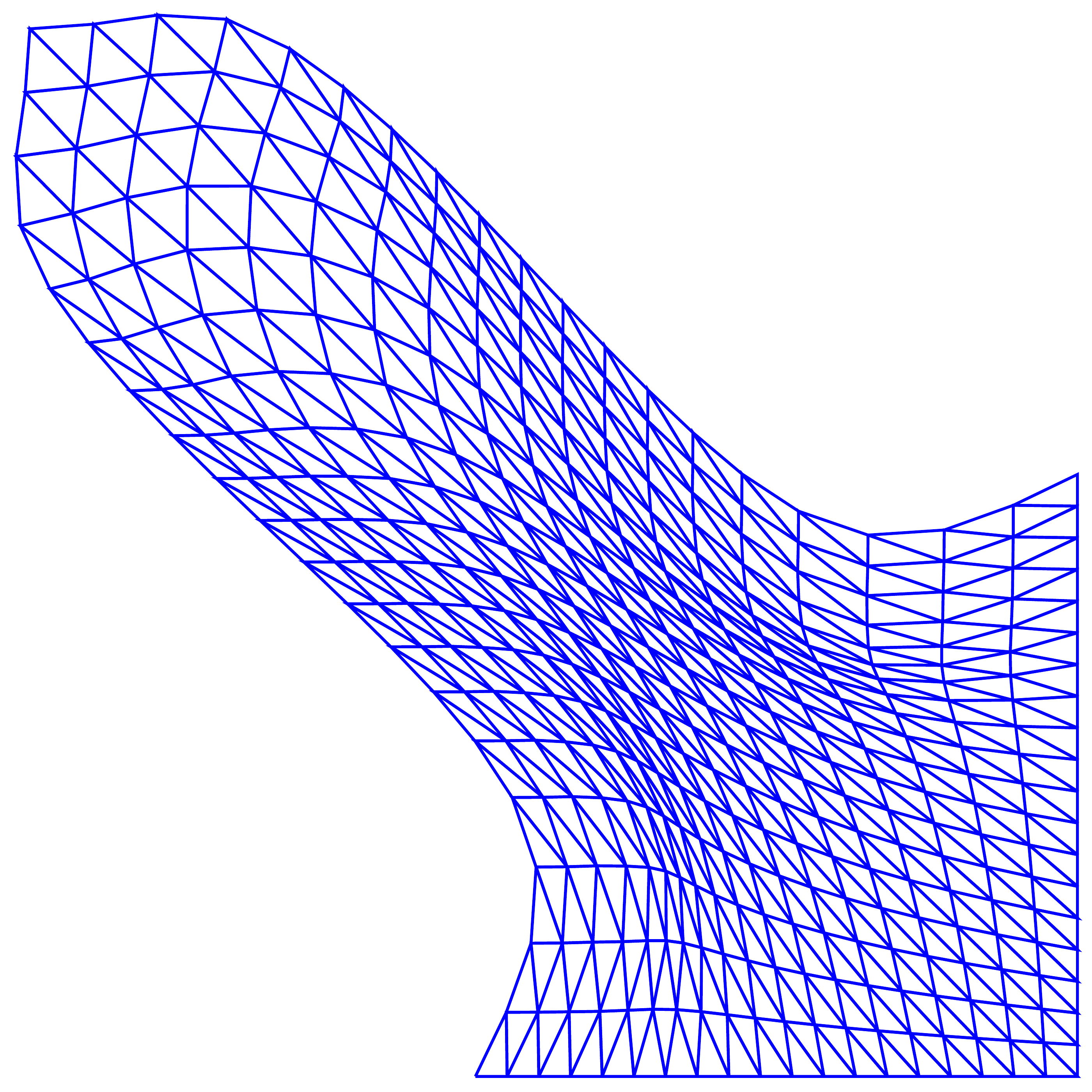}
    \end{subfigure}
    \begin{subfigure}[b]{0.3\textwidth}
        \centering
        \includegraphics[width=\textwidth]{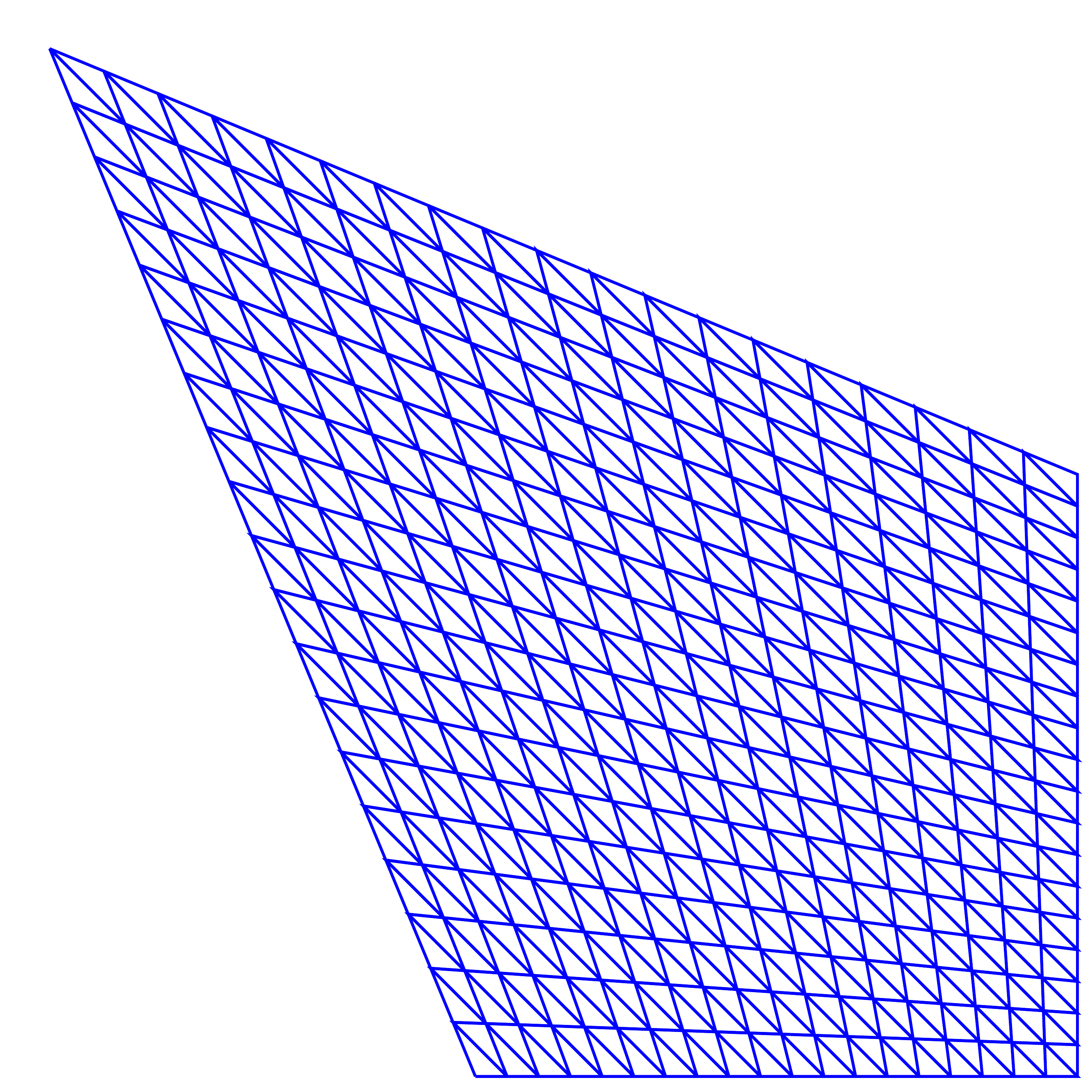}
    \end{subfigure}
    \begin{subfigure}[b]{0.3\textwidth}
        \centering
        \includegraphics[width=\textwidth]{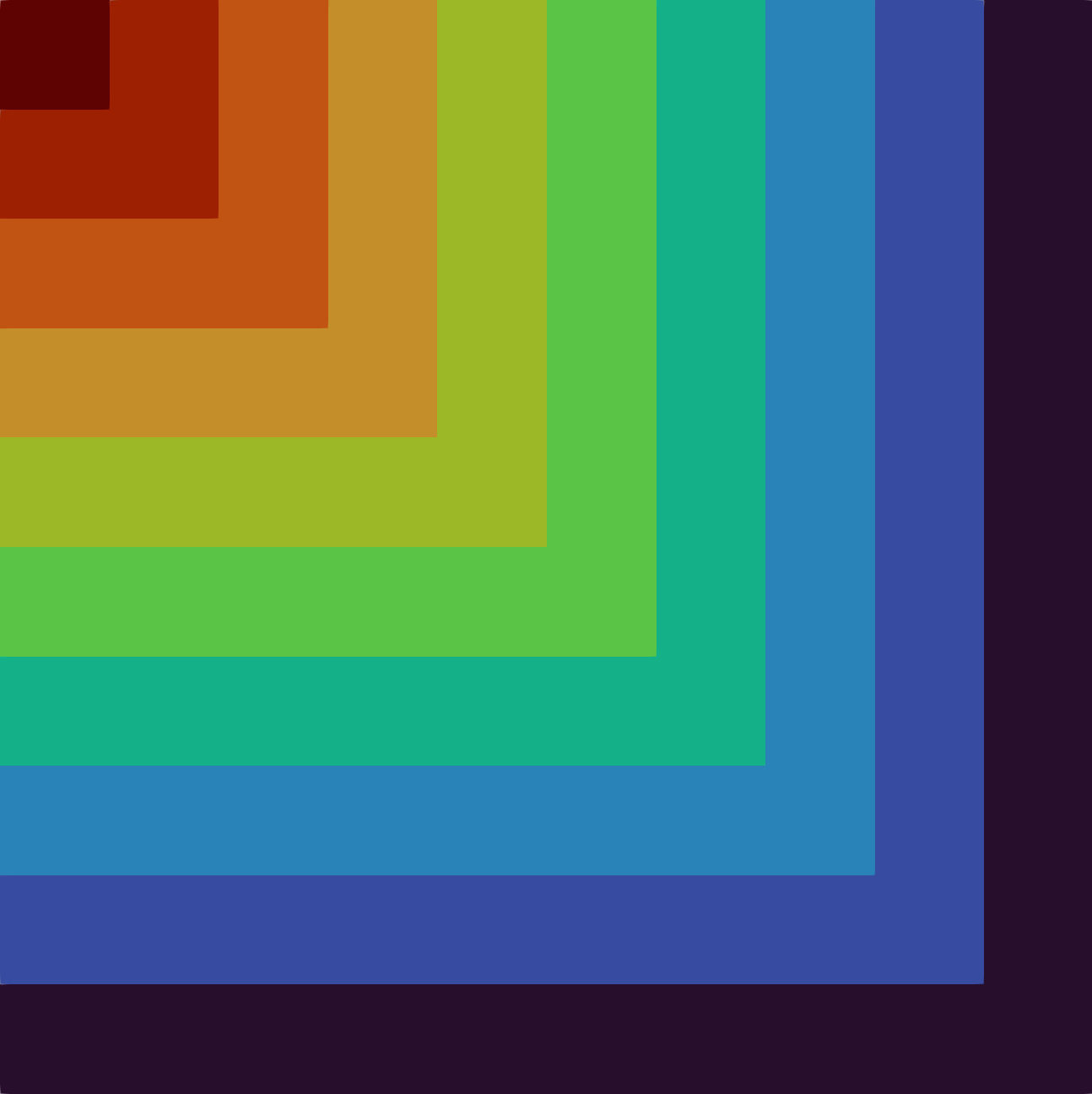}
    \end{subfigure}
    \begin{subfigure}[b]{0.3\textwidth}
        \centering
        \includegraphics[width=\textwidth]{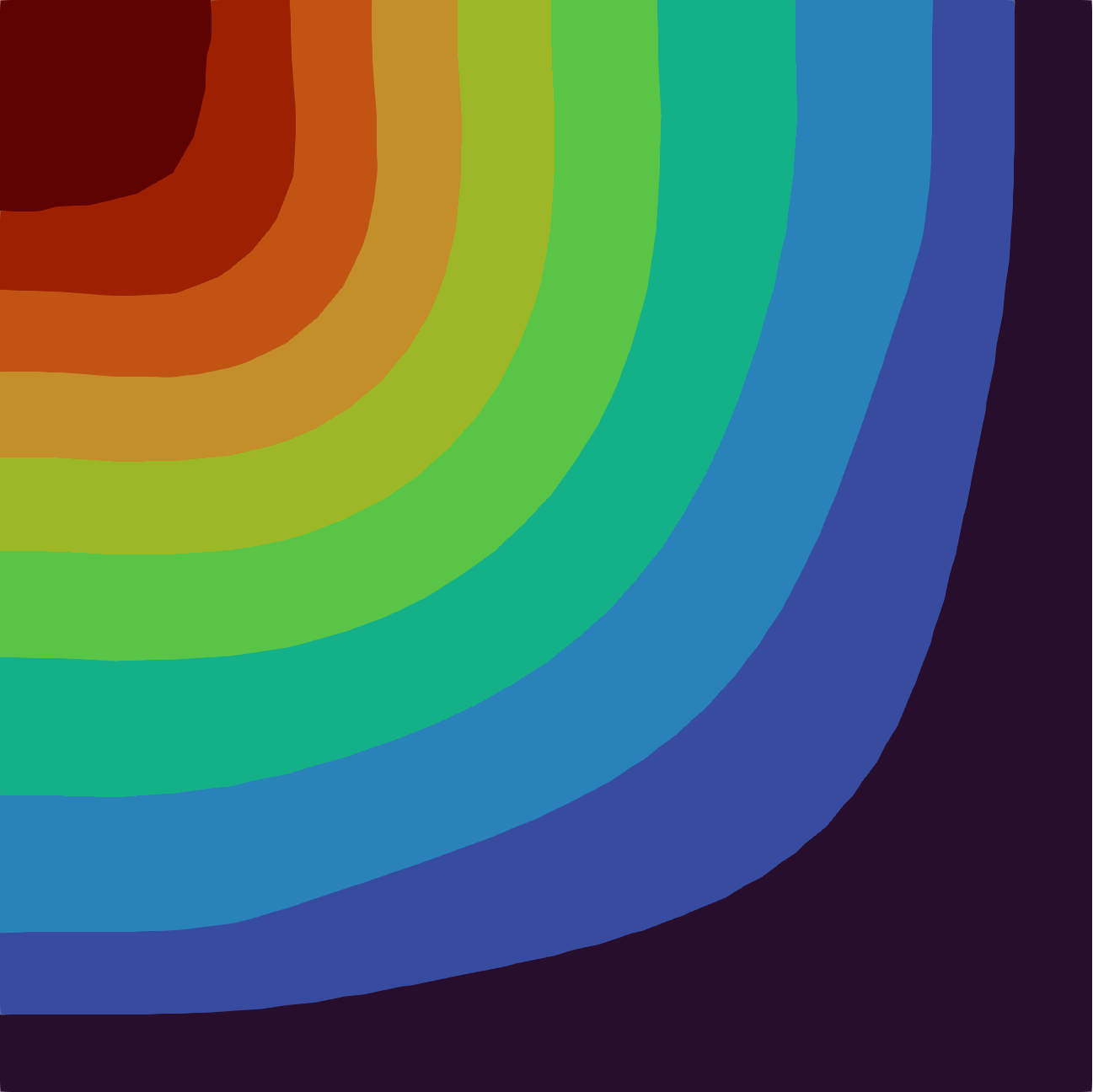}
    \end{subfigure}
    \begin{subfigure}[b]{0.3\textwidth}
        \centering
        \includegraphics[width=\textwidth]{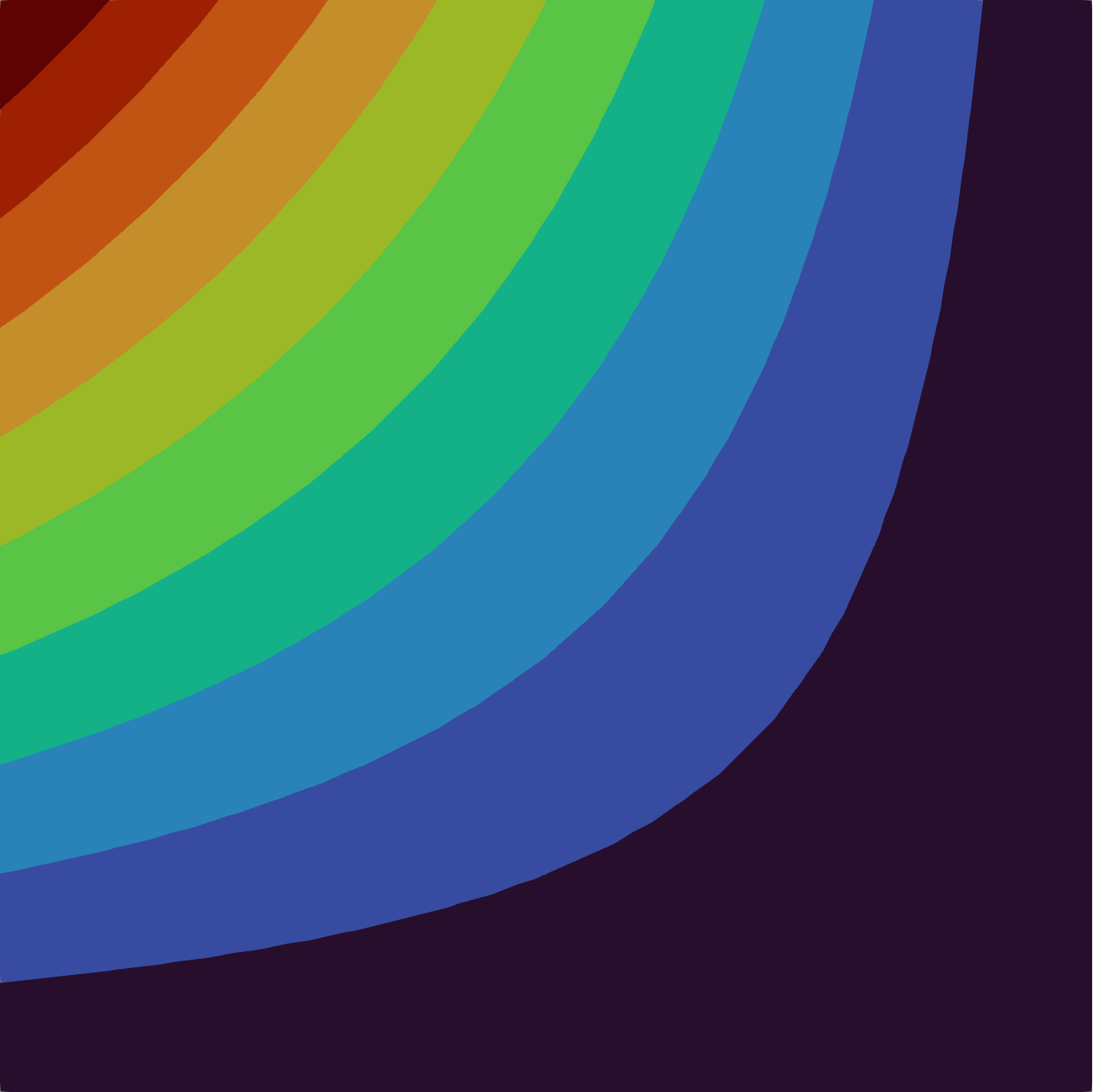}
    \end{subfigure}
    \begin{subfigure}[b]{0.3\textwidth}
        \centering
        \includegraphics[width=\textwidth]{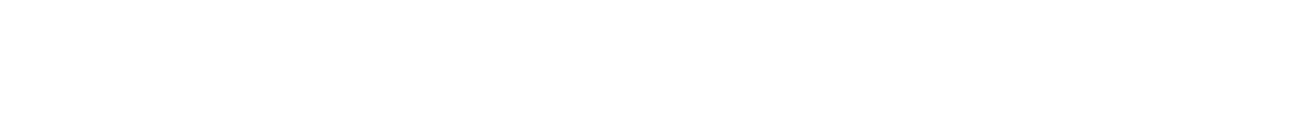}
        \caption{Expected limit.}
        \label{fig:InfSolArtificial}
    \end{subfigure}
    \begin{subfigure}[b]{0.3\textwidth}
        \centering
        \includegraphics[width=\textwidth]{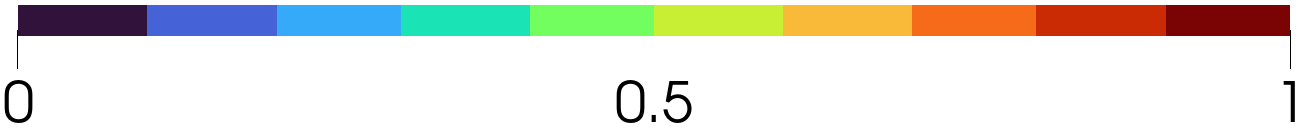}
        \caption{Inner norm $\lVert \cdot \rVert_2$.}
        \label{fig:InfSolLoisel}
    \end{subfigure}
    \begin{subfigure}[b]{0.3\textwidth}
        \centering
        \includegraphics[width=\textwidth]{images/DeformationField_LegendEmpty.png}
        \caption{Inner norm $\lVert \cdot \rVert_\infty$.}
        \label{fig:InfSolSupSup}
    \end{subfigure}
    \captionsetup{format=hang}
    \caption{Deformation of a square domain with $h(x) = (\sin(2\pi x_1) - \sin(2\pi x_2)) \cdot \eta$ prescribed on the left and upper boundary for different different algorithmic realizations of the $\infty$-Laplacian.}
    \label{fig:DifferentInfSols}
\end{figure}

\noindent Additionally to the algorithm for the finite setting, which we modified in section \ref{sec:HighOrderDescent}, an algorithm for the limit problem is proposed in \cite{loisel2020}.
There, the limit of problem (\ref{eq:pLaplaceZeroTraceMinimization}) is formulated as
\begin{equation}
    \underset{u\in \mathcal{U}^{\infty}}{\mathrm{arg\;min}}\; J_\infty(u) = \underbrace{\sup_{x \in \Omega} \lVert \nabla (u+g) \rVert}_{=: \lVert u+g \rVert_{X^\infty(\Omega)}} - \int_{\Gamma} h u \;\mathrm{d}\Gamma - \int_{\Omega} f u \;\mathrm{d} x.
    \label{eq:InfLaplaceMinimization}
\end{equation}
First, note that we leave the interpretation of the interior norm in the first term open for now.
While there is no actual derivation of the formulation given, we can state some arguments to consider it.
Interpreting the primary term in equation (\ref{eq:pLaplaceZeroTraceMinimization}) as $\lVert u+g \rVert_{X^p(\Omega)}^p$ one would obtain this primary term as $\lVert u+g \rVert_{X^\infty(\Omega)}$ in the same way classical limits for $\mathrm{L}^p$-norms are constructed.
By the theory of Lipschitz extensions \cite{jensen1993}, one can understand the $\infty$-Laplacian as the minimization of the sup norm of the gradient, given further justification to the idea.\\

\noindent We will not give the construction of the algorithm here again, since it is similar to the finite setting and the extensions are done in the same manner.
However, it is interesting that the reformulated discrete problem (\ref{eq:ReformulatedFiniteMinimization}) essentially becomes a problem for the $1$-Laplacian over the subspace of constant $s$ meaning that the Lipschitz constant on each element is bounded uniformly and the bound is minimized.
On the Neumann boundary this approach would result in $v\rvert_{\Gamma} = 0$ and thus we add the constraint $s \geq 1$ in order to obtain the desired slope on the boundary.\\

\noindent From the sequence of solutions for finite $p$ and the observations in 1D combined with the theoretical definition of the AMLE, we expect the limit artificially shown in figure \ref{fig:InfSolArtificial}.
Figure \ref{fig:InfSolLoisel} shows the result for the choice of the inner norm $\lVert \cdot \rVert_2$ in equation (\ref{eq:InfLaplaceMinimization}) as proposed in \cite{loisel2020}.
It overshoots the intended tip slightly, yielding an improvement to the solution for $p=15$ from figure \ref{fig:FiniteDeformations}.
But it is not able to resolve the intended free boundaries, leading to worse results than the finite setting.
However, the grid quality at those boundaries remains good and instead the center becomes significantly worse.
By the method of manufactured solutions, we can also show that it is not able to yield the unique solution for a reference problems with the well-known viscosity solution $v=x_1^{4/3}-x_2^{4/3}$ \cite{lindqvist2016}.
Another approach is choosing the supremum norm also for the inner norm, which we computed in figure \ref{fig:InfSolSupSup}.
While this choice yields the correct outer boundary, it still does not deform the interior mesh uniformly.\\
\section{Conclusion}
\label{sec:Conclusion}

We added support for Neumann boundary conditions to a present algorithm for the scalar $p$-Laplacian and proved that the theoretical estimate on the required Newton steps remains polynomial.
Further, we constructed the extension of the algorithm to vector-valued problems and performed numerical experiments including validation and shape deformations.
The results demonstrate that the extension is indeed applicable to problems occurring in $p$-harmonic shape optimization and yields solutions for higher-order $p>5$ without iterating over $p$.
Those provide further improvements in terms of preserved mesh quality and obtained boundary shape.
However, we saw that results obtained for the $\infty$-Laplacian are significantly different and even high-order solutions do not yield a sufficient approximation, especially since they depend on the magnitude of the source term.
First experiments with a modified algorithm to solve the $\infty$-Laplacian problem directly did not yield the desired results, but small changes could already achieve improvements, allowing the idea to be considered further.\\

\noindent For future research it remains to construct a proper algorithm for the limit problem.
On the other hand, the results for finite $p$ may be applied in shape optimization problems potentially including 3D settings.
This includes modifications to the implementation for high-performance computer architecture and analysis considering scalability.

\section*{Acknowledgment}
The authors acknowledge the support by the Deutsche Forschungsgemeinschaft (DFG) within the Research Training Group GRK 2583 ``Modeling, Simulation and Optimization of Fluid Dynamic Applications''.

\printbibliography

@incollection{allaire2021,
    title = {{Shape and topology optimization}},
    author = {Allaire, Gregoire and Dapogny, Charles and Jouve, Francois},
    year = {2021},
    booktitle={Geometrix Partial Differential Equations - Part II},
    series={Handbook of Numerical Analysis},
    publisher = {Elsevier},
    volume={22},
    pages={1-132},
    doi = {10.1016/bs.hna.2020.10.004}
}

@article{aronsson2004,
    title = {{A tour of the theory of Absolutely Minimizing Functions}},
    author = {Aronsson, Gunnar and Crandall, Michael and Juutinen, Petri},
    year = {2004},
    pages = {},
    volume = {41},
    journal = {Bulletin of The American Mathematical Society},
    doi = {10.1090/S0273-0979-04-01035-3}
}

@article{bonder2003,
    title = {{Uniform bounds for the best Sobolev trace constant}},
    author = {Fernández Bonder, Julián and Rossi, Julio and Ferreira, Raúl},
    year = {2003},
    pages = {181-192},
    volume = {3},
    journal = {Advanced Nonlinear Studies},
    doi = {10.1515/ans-2003-0202}
}

@article{deckelnick2021,
    title={{A novel $W^{1,\infty}$ approach to shape optimisation with Lipschitz domains}}, 
    author={Deckelnick, Klaus and Herbert, Philip J. and Hinze, Michael},
    year={2022},
    volume={28},
    journal={ESAIM: Control, Optimisation and Calculus of Variations},
    publisher={EDP Sciences},
    doi={10.1051/cocv/2021108},
}

@book{delfour2011,
    title = {{Shapes and Geometries: Metrics, Analysis, Differential Calculus, and Optimization}},
    series = {Advances in Design and Control},
    volume = {22},
    author = {Delfour, Michel C. and Zolesio, Jean-Paul},
    publisher = {SIAM},
    year = {2011},
    edition = {Second Edition},
    doi = {10.1137/1.9780898719826}
}

@book{evans2015,
    title={{Measure Theory and Fine Properties of Functions}},
    author={Evans, Lawrence C. and Gariepy, Ronald F.},
    year={2015},
    publisher={Chapman and Hall},
    edition={Revised Edition},
    doi = {10.1201/b18333}
}

@article{ishii2005,
    author = {Ishii, Hitoshi and Loreti, Paola},
    title = {{Limits of Solutions of p-Laplace Equations as p Goes to Infinity and Related Variational Problems}},
    journal = {SIAM Journal on Mathematical Analysis},
    volume = {37},
    number = {2},
    pages = {411-437},
    year = {2005},
    doi = {10.1137/S0036141004432827}
}

@article{jensen1993,
    author = {Jensen, Robert},
    title = {{Uniqueness of Lipschitz extensions: Minimizing the sup norm of the gradient}},
    journal = {Archive for Rational Mechanics and Analysis},
    year = {1993},
    volume = {123},
    number = {1},
    pages = {51-74},
    doi = {10.1007/BF00386368}
}

@article{lindqvist2019,
    title = {{Notes on the Stationary p-Laplace Equation}},
    author = {Lindqvist, Peter},
    series = {SpringerBriefs in Mathematics},
    publisher = {Springer Initernational Publishing},
    year = {2019},
    doi = {0.1007/978-3-030-14501-9}
}

@book{lindqvist2016,
    title = {{Notes on the Infinity Laplace Equation}},
    author = {Lindqvist, Peter},
    series = {SpringerBriefs in Mathematics},
    publisher = {Springer Initernational Publishing},
    year = {2016},
    doi = {10.1007/978-3-319-31532-4}
}

@article{loisel2020,
    title={{Efficient algorithms for solving the p-Laplacian in polynomial time}},
    volume={146},
    number = {2},
    DOI={10.1007/s00211-020-01141-z},
    journal={Numerische Mathematik},
    publisher={Springer},
    author={Loisel, Sébastien},
    year={2020},
    pages={369–400}
}

@misc{minfem2020,
    author={Siebenborn, Martin},
    title = {{MinFEM}},
    howpublished = {\url{https://github.com/msiebenborn/MinFEM.jl}},
    year={2022}
}

@article{mueller2021,
    title={{A Novel $p$-Harmonic Descent Approach Applied to Fluid Dynamic Shape Optimization}}, 
    author={Müller, Peter Marvin and Kühl, Niklas and Siebenborn, Martin and Deckelnick, Klaus and Hinze, Michael and Rung, Thomas},
    year={2021},
    journal={Structural and Mulitdisciplinary Optimization},
    publisher={Springer},
    doi={10.1007/s00158-021-03030-x}
}

@book{nesterov2001,
    title = {{Interior-point polynomial algorithms in convex programming}},
    series = {SIAM studies in applied mathematics},
    author = {Nesterov, Yurii and Nemirovskij, Arkadij S.},
    publisher = {SIAM},
    year = {2001},
    edition = {Third Edition},
    doi = {10.1137/1.9781611970791}
}

@book{nesterov2004,
    title = {{Introductory lectures on convex optimization : a basic course}},
    series = {Applied optimization},
    author = {Nesterov, Yurii},
    publisher = {Kluwer Acad. Publ.},
    year = {2004},
    doi = {10.1007/978-1-4419-8853-9}
}

@article{salari2000,
    author = {Salari, Kambiz and Knupp, Patrick},
    title = {{Code Verification by the Method of Manufactured Solutions}},
    journal = {Sandia Report},
    publisher = {Sandia National Laboratories},
    year = {2000},
    doi = {10.2172/759450}
}

@article{schulz2016,
   title={{Computational Comparison of Surface Metrics for PDE Constrained Shape Optimization}},
   volume={16},
   doi={10.1515/cmam-2016-0009},
   number={3},
   journal={Computational Methods in Applied Mathematics},
   publisher={Walter de Gruyter GmbH},
   author={Schulz, Volker and Siebenborn, Martin},
   year={2016},
   pages={485–496}
}

@book{sokolowski1992,
    title = {{Introduction to shape optimization : shape sensitivity analysis}},
    series = {Springer series in computational mathematics},
    volume = {16},
    author = {Sokolowski, Jan and Zolesio, Jean-Paul},
    publisher = {Springer},
    year = {1992},
    doi = {10.1007/978-3-642-58106-9}
}
\end{document}